\documentclass[12pt]{amsart}
\usepackage{amsmath,amssymb,amsbsy,amsfonts,latexsym,amsopn,amstext,cite,
                                               amsxtra,euscript,amscd,bm}
\usepackage{url}

\usepackage{mathrsfs}

\usepackage{color}
\usepackage[colorlinks,linkcolor=blue,anchorcolor=blue,citecolor=blue,backref=page]{hyperref}
\usepackage{color}
\usepackage{graphics,epsfig}
\usepackage{graphicx}
\usepackage{float}
\usepackage{epstopdf}
\hypersetup{breaklinks=true}

\usepackage[np]{numprint}
\npdecimalsign{\ensuremath{.}}

\usepackage{bibentry}

\usepackage[english]{babel}
\usepackage{mathtools}
\usepackage{todonotes}
\usepackage{url}
\usepackage[colorlinks,linkcolor=blue,anchorcolor=blue,citecolor=blue,backref=page]{hyperref}

\usepackage[norefs,nocites]{refcheck}

\DeclareMathOperator*\spt{spt}

\DeclareMathOperator*\uplim{\overline{lim}}

\begin{document}

\newtheorem{theorem}{Theorem}
\newtheorem{lemma}[theorem]{Lemma}
\newtheorem{claim}[theorem]{Claim}
\newtheorem{cor}[theorem]{Corollary}
\newtheorem{conj}[theorem]{Conjecture}
\newtheorem{prop}[theorem]{Proposition}
\newtheorem{definition}[theorem]{Definition}
\newtheorem{question}[theorem]{Question}
\newtheorem{example}[theorem]{Example}
\newcommand{\hh}{{{\mathrm h}}}
\newtheorem{remark}[theorem]{Remark}

\numberwithin{equation}{section}
\numberwithin{theorem}{section}
\numberwithin{table}{section}
\numberwithin{figure}{section}

\def\sssum{\mathop{\sum\!\sum\!\sum}}
\def\ssum{\mathop{\sum\ldots \sum}}
\def\iint{\mathop{\int\ldots \int}}

\newcommand{\diam}{\operatorname{diam}}

\def\squareforqed{\hbox{\rlap{$\sqcap$}$\sqcup$}}
\def\qed{\ifmmode\squareforqed\else{\unskip\nobreak\hfil
\penalty50\hskip1em \nobreak\hfil\squareforqed
\parfillskip=0pt\finalhyphendemerits=0\endgraf}\fi}

\newfont{\teneufm}{eufm10}
\newfont{\seveneufm}{eufm7}
\newfont{\fiveeufm}{eufm5}
%
%
\newfam\eufmfam
     \textfont\eufmfam=\teneufm
\scriptfont\eufmfam=\seveneufm
     \scriptscriptfont\eufmfam=\fiveeufm
%
%
\def\frak#1{{\fam\eufmfam\relax#1}}
\def\muS{\mu_{\mathsf {S}}}
\def\muM{\mu_{\mathsf {M}}}
\def\muo{\mu_{\bomega}}

\def\muMd{\mu_{\mathsf {M, \delta}}}

\newcommand{\bflambda}{{\boldsymbol{\lambda}}}
\newcommand{\bfmu}{{\boldsymbol{\mu}}}
\newcommand{\bfxi}{{\boldsymbol{\eta}}}
\newcommand{\bfrho}{{\boldsymbol{\rho}}}

\def\eps{\varepsilon}

\def\fI{\mathfrak I}
\def\fK{\mathfrak K}
\def\fT{\mathfrak{T}}
\def\fL{\mathfrak L}
\def\fR{\mathfrak R}

\def\fA{{\mathfrak A}}
\def\fB{{\mathfrak B}}
\def\fC{{\mathfrak C}}
\def\fM{{\mathfrak M}}
\def\fS{{\mathfrak  S}}
\def\fU{{\mathfrak U}}

\def\T {\mathsf {T}}
\def\Tor{\mathsf{T}_d}
\def\Tore{\widetilde{\mathrm{T}}_{d} }

\def\sM {\mathsf {M}}

\def\ss{\mathsf {s}}

\def\Kmnd{\cK_d(m,n)}
\def\Kmnp{\cK_p(m,n)}
\def\Kmnq{\cK_q(m,n)}

\def \balpha{\bm{\alpha}}
\def \bbeta{\bm{\beta}}
\def \bgamma{\bm{\gamma}}
\def \bdelta{\bm{\delta}}
\def \bzeta{\bm{\zeta}}
\def \blambda{\bm{\lambda}}
\def \bchi{\bm{\chi}}
\def \bphi{\bm{\varphi}}
\def \bpsi{\bm{\psi}}
\def \bnu{\bm{\nu}}
\def \bxi{\bm{\xi}}
\def \bomega{\bm{\omega}}

\def \bell{\bm{\ell}}

\def\eqref#1{(\ref{#1})}

\def\vec#1{\mathbf{#1}}

\newcommand{\abs}[1]{\left| #1 \right|}

\def\Sp{\mathbb{S}}

\def\Zq{\mathbb{Z}_q}
\def\Zqx{\mathbb{Z}_q^*}
\def\Zd{\mathbb{Z}_d}
\def\Zdx{\mathbb{Z}_d^*}
\def\Zf{\mathbb{Z}_f}
\def\Zfx{\mathbb{Z}_f^*}
\def\Zp{\mathbb{Z}_p}
\def\Zpx{\mathbb{Z}_p^*}
\def\cM{\mathcal M}
\def\cE{\mathcal E}
\def\cH{\mathcal H}

\def\le{\leqslant}

\def\ge{\geqslant}

\def\sfB{\mathsf {B}}
\def\sfC{\mathsf {C}}
\def\L{\mathsf {L}}
\def\FF{\mathsf {F}}

\def\sE {\mathscr{E}}
\def\sS {\mathscr{S}}

\def\cA{{\mathcal A}}
\def\cB{{\mathcal B}}
\def\cC{{\mathcal C}}
\def\cD{{\mathcal D}}
\def\cE{{\mathcal E}}
\def\cF{{\mathcal F}}
\def\cG{{\mathcal G}}
\def\cH{{\mathcal H}}
\def\cI{{\mathcal I}}
\def\cJ{{\mathcal J}}
\def\cK{{\mathcal K}}
\def\cL{{\mathcal L}}
\def\cM{{\mathcal M}}
\def\cN{{\mathcal N}}
\def\cO{{\mathcal O}}
\def\cP{{\mathcal P}}
\def\cQ{{\mathcal Q}}
\def\cR{{\mathcal R}}
\def\cS{{\mathcal S}}
\def\cT{{\mathcal T}}
\def\cU{{\mathcal U}}
\def\cV{{\mathcal V}}
\def\cW{{\mathcal W}}
\def\cX{{\mathcal X}}
\def\cY{{\mathcal Y}}
\def\cZ{{\mathcal Z}}
\newcommand{\rmod}[1]{\: \mbox{mod} \: #1}

\def\cg{{\mathcal g}}

\def\vt{\mathbf t}
\def\vy{\mathbf y}
\def\vr{\mathbf r}
\def\vx{\mathbf x}
\def\va{\mathbf a}
\def\vb{\mathbf b}
\def\vc{\mathbf c}
\def\ve{\mathbf e}
\def\vh{\mathbf h}
\def\vk{\mathbf k}
\def\vm{\mathbf m}
\def\vz{\mathbf z}
\def\vu{\mathbf u}
\def\vv{\mathbf v}
\def\v0{\mathbf 0}

\def\e{{\mathbf{\,e}}}
\def\ep{{\mathbf{\,e}}_p}
\def\eq{{\mathbf{\,e}}_q}

\def\Tr{{\mathrm{Tr}}}
\def\Nm{{\mathrm{Nm}}}

 \def\SS{{\mathbf{S}}}

\def\lcm{{\mathrm{lcm}}}

 \def\0{{\mathbf{0}}}

\def\({\left(}
\def\){\right)}
\def\l|{\left|}
\def\r|{\right|}
\def\fl#1{\left\lfloor#1\right\rfloor}
\def\rf#1{\left\lceil#1\right\rceil}

\def\cl#1{\left\lfloor#1\right\rceil}
\def\sumstar#1{\mathop{\sum\vphantom|^{\!\!*}\,}_{#1}}

\def\mand{\qquad \mbox{and} \qquad}

\def\tblue#1{\begin{color}{blue}{{#1}}\end{color}}




\hyphenation{re-pub-lished}

\mathsurround=1pt

\def\bfdefault{b}

\def \F{{\mathbb F}}
\def \K{{\mathbb K}}
\def \N{{\mathbb N}}
\def \Z{{\mathbb Z}}
\def \P{{\mathbb P}}
\def \Q{{\mathbb Q}}
\def \R{{\mathbb R}}
\def \C{{\mathbb C}}
\def\Fp{\F_p}
\def \fp{\Fp^*}

 \def \xbar{\overline x}

\title[Restricted Weyl sums]{Restricted mean value theorems and 
metric theory of restricted Weyl sums}

\author[C. Chen] {Changhao Chen}

\address{Department of Pure Mathematics, University of New South Wales,
Sydney, NSW 2052, Australia}
\email{changhao.chenm@gmail.com}

 \author[I. E. Shparlinski] {Igor E. Shparlinski}

\address{Department of Pure Mathematics, University of New South Wales,
Sydney, NSW 2052, Australia}
\email{igor.shparlinski@unsw.edu.au}

\begin{abstract} 
We study  an apparently new question about the behaviour of  
Weyl sums on a subset $\cX\subseteq [0,1)^d$ with a natural measure $\mu$  on $\cX$. For certain measure spaces $(\cX, \mu)$ we obtain non-trivial bounds for the mean values of the Weyl sums, and for $\mu$-almost all points of $\cX$ the Weyl sums satisfy the square root cancellation law. Moreover we characterise the size of the exceptional sets  in terms of Hausdorff dimension. 

Finally, we derive variants of the Vinogradov mean value theorem averaging over 
measure spaces $(\cX, \mu)$.
We obtain general results, which we refine for some special spaces $\cX$ such as spheres, moment curves and
line segments.  
 \end{abstract}


\maketitle

\tableofcontents

\section{Introduction}

\subsection{Background}

For an integer $d \ge 2$, let $\Tor = (\R/\Z)^d$ be the  $d$-dimensional unit torus. 

For  a vector $\vx = (x_1, \ldots, x_d)\in \Tor$ and $N \in\N$, we consider the exponential   
sums
$$
S_d(\vx; N)=\sum_{n=1}^{N}\e\(x_1 n+\ldots +x_d n^{d} \), 
$$
which  are commonly called {\it Weyl sums\/}, where   throughout  the paper we denote $\e(x)=\exp(2 \pi i x)$.

Weyl sums, introduced by Weyl~\cite{Weyl} as a tool to  investigate  the  distribution 
of fractional parts of real polynomials (see also~\cite{Baker}) also  
appear in a broad spectrum  of other number theoretic problems.
For example,  they play a crucial role in  
estimating the {\it zero-free region} of the {\it Riemann zeta-function} and thus in turn obtaining a sharp form of  the {\it prime number theorem}, see~\cite[Section~8.5]{IwKow},  
in the {\it Waring problem}, see~\cite[Section~20.2]{IwKow}, in bounding 
 short character sums modulo highly composite numbers~\cite[Section~12.6]{IwKow} and many others.

However, despite more than a century long history of estimating such sums, the behaviour of individual sums is not well understood, see~\cite{Brud,BD}.

The following best known bound is a direct implication of the current form of the {\it Vinogradov mean value theorem} from~\cite{BDG, Wool2} (see also~\eqref{eq:MVT} below)   and is given in~\cite[Theorem~5]{Bourg}.   Let $\vx = (x_1, \ldots, x_d)  \in \Tor$ be such that for some $\nu$ with $2 \le \nu\le d $ and some positive integers $a$ and $q$ with $\gcd(a,q)=1$  we have
$$
\left| x_\nu - \frac{a}{q}\right| \le \frac{1}{q^2}. 
$$
Then for any $\varepsilon>0$ there exists a constant $C(\varepsilon)$ such that  
$$
|S_d(\vx; N)| \le C(\varepsilon)N^{1+\varepsilon} \(q^{-1} + N^{-1} + qN^{-\nu}\)^{\frac{1}{d(d-1)}}.
$$

On the other hand, thanks to recent striking results   of  Bourgain, Demeter and Guth~\cite{BDG} (for $d \geqslant 4$) 
and Wooley~\cite{Wool2} (for $d=3$)  (see also~\cite{Wool5}), for  any integer $s \ge 1$, for the 
$2s$-power mean value  of  $S_d(\vx; N)$ we have   
\begin{equation}
\label{eq:MVT}
\int_{\Tor} |S_d(\vx; N)|^{2s}d\vx \leqslant  N^{s +o(1)} +  N^{2s - s(d)+o(1)}, \qquad N \to  \infty,
\end{equation} 
where  
$$
s(d)=d(d+1)/2, 
$$
which is the best possible  form  
of  the Vinogradov mean value theorem. 
In particular 
$$
\int_{\Tor} |S_d(\vx; N)|^{2s(d)}d\vx \leqslant  N^{s(d)+o(1)}, \qquad N \to  \infty. 
$$

\subsection{Previous results and questions}

We first outline  some results concerning  the metric theory  of Weyl sums on $\T_d$.
The metric theory means that we study the properties of Weyl sums  which hold 
for almost all points with respect to the Lebesgue or some other measures. Moreover one also characterise the size of the exceptional sets (outside of the almost all) in terms of Hausdorff dimension.   

We remark that the topic here, the metric theory of Weyl sums, is not the same as the topics in
the {\it metric theory of numbers\/}, see, for instance,~\cite{Harm}. 
However they are certainly related to each other.

We say that some property holds for {\it almost all $\vx \in \Tor$\/}  if it holds for a set 
 $\cX \subseteq \Tor$ of  Lebesgue measure  $\lambda(\cX) = 1$.  

For $d=2$, Fedotov and Klopp~\cite[Theorem~0.1]{FK} 
give  the following optimal lower and upper bounds. 
Suppose that  $\{g(n)\}_{n=1}^{\infty}$ is a non-decreasing sequence of positive numbers. Then for almost all  $\vx\in \T_2$  one has 
$$
\uplim_{N\to  \infty} \frac{  \left |S_2(\vx; N)\right |}{\sqrt{N} g(\ln  N)}<\infty \quad  \Longleftrightarrow \quad 
 \sum_{n=1}^{\infty} \frac{1}{ g(n)^{6}} <\infty.
$$

For $d\ge 3$, the authors~\cite[Appendix~A]{ChSh-On} have
shown that for almost all $\vx\in \T_d$ 
\begin{equation}
\label{eq:1/2}
|S_{d}(\vx; N)|\le N^{1/2+o(1)}, \qquad N\rightarrow \infty.
\end{equation}
One may conjecture that this is the best possible bound, see~\cite[Conjecture~1.1]{ChSh-On}.

Let $\va=(a_n)_{n=1}^{\infty}$ be a sequence of complex weights and denote  
\begin{equation}
\label{eq:S_ad}
S_{\va, d}(\vx; N)=\sum_{n=1}^{N}a_n \e(x_1n+\ldots+x_d n^d).
\end{equation}
Extending~~\eqref{eq:1/2},  the authors~\cite[Corollary~2.2]{ChSh-New} have shown  that for any complex weights $\va=(a_n)_{n=1}^{\infty}$  with $a_n=n^{o(1)}$ one has that for almost all $\vx\in \T_d$, 
\begin{equation}
\label{eq:1/2-weight}
|S_{\va, d}(\vx; N)|\le N^{1/2+o(1)}, \qquad N\rightarrow \infty. 
\end{equation}   

From the almost all results in~\eqref{eq:1/2} and~\eqref{eq:1/2-weight} one may ask  how ``large" are the exceptional sets. For this purpose we introduce following notation.  For  $0<\alpha<1$ and integer $d\ge 2$, we consider the set  
\begin{equation} 
\label{eq:exceptional}
\begin{split}
\cE_{\va, d,\alpha}=\{\vx\in \Tor:&~|S_{\va, d}( \vx; N)|\ge N^{\alpha}\\
 & \qquad  \text{ for infinity many }N\in \N\},
\end{split}
\end{equation} 
and call it the {\it exceptional set\/}.  
If $\va=\ve=(1)_{n=1}^{\infty}$ we just write  
$$
\cE_{d, \alpha}=\cE_{\ve, d, \alpha}.
$$
Using this notation we may say that for any $1/2<\alpha<1$ the set $\cE_{d, \alpha}$ has zero Lebesgue measure.

For sets of  Lebesgue measure zero, it is common to use  the {\it Hausdorff dimension\/} to describe their size; for the properties of the Hausdorff dimension and its applications we refer to~\cite{Falconer}.  We recall that for $\cU \subseteq \R^{d}$
$$
\diam \cU = \sup\{\| \vu-\vv\|:~\vu,\vv \in \cU\}, 
$$ 
where $\|\mathbf{w}\|$ is the Euclidean norm in $\R^{d}$. 

\begin{definition} 
The  Hausdorff dimension of a set $\cA\subseteq \R^{d}$ is defined as 
\begin{align*}
\dim \cA=\inf\Bigl\{s>0:~\forall &\,  \eps>0,~\exists \, \{ \cU_i \}_{i=1}^{\infty}, \ \cU_i \subseteq \R^{d},\\
&  \text{such that } \cA\subseteq \bigcup_{i=1}^{\infty} \cU_i \text{ and } \sum_{i=1}^{\infty}\(\diam\cU_i\)^{s}<\eps \Bigr\}.
\end{align*}
\end{definition}

We note that the authors~\cite{ChSh-On} have obtained a  lower bound of the Hausdorff dimension of $\cE_{d, \alpha}$. Among other things, it is shown in~\cite{ChSh-On} for any $\alpha\in (0,1)$ one has 
$$
\dim \cE_{d, \alpha}\ge \ell(d, \alpha)
$$
with some explicit function $\ell(d, \alpha)>0$.

Furthermore, the authors~\cite{ChSh-Hausdorff} have  given  a non-trivial  upper bound for $\cE_{d, \alpha}$. More precisely, we have 
\begin{equation}
\label{eq:upper-u}
\dim \cE_{d, \alpha}\le \mathfrak{u}(d, \alpha),
\end{equation}
where 
\begin{equation}
\label{eq:def u} 
\mathfrak{u}(d, \alpha)=\min_{k=0, \ldots, d-1} \frac{(2d^{2}+4d)(1-\alpha)+k(k+1)}{4-2\alpha+2k}.
\end{equation}
The bound~\eqref{eq:def u} has some  interesting implications. For instance for any $\alpha \in (1/2, 1)$ we have 
$$
\dim \cE_{d, \alpha}<d.
$$
Moreover, if $\alpha\rightarrow 1$ then 
$u(d, \alpha)\rightarrow 0$. 
Indeed it is expected that if $\alpha$ becomes large then the set $\cE_{d, \alpha}$ becomes small. We refer to~\cite{ChSh-Hausdorff} for more details.

Furthermore,  as a counterpart to~\eqref{eq:upper-u}, we remark that we expect $\dim \cE_{d, \alpha} = d$ for $\alpha\in (0,1/2]$, see also~\cite{ChSh-On, ChSh-Hausdorff}. On the other hand, we do not have any 
plausible conjecture about the exact behaviour of  $\dim \cE_{d, \alpha}$ for $\alpha\in (1/2, 1)$.

\subsection{Average values and the metric properties of restricted Weyl sums}  
The goal here is to  investigate the Weyl sums over some subset $\cX\subseteq \T_d$ with some natural measure on $\cX$. Restricted type 
Such restriction problem arise  in many other areas of mathematics. These include  {\it Diophantine approximation on manifolds}, see~\cite{BVVZ}, {\it Fourier restriction problems}, see~\cite[Chapter~19]{Mattila2015}, the {\it restricted families of projections\/}, see~\cite[Section~5.4]{Mattila2015}, and {\it discrete Fourier restriction\/},  see~\cite{HH1, HH2, HuWo, LD, Wool4}.   

We recall that the {\it support} $\spt\mu$ of a measure $\mu$ on $\R^d$ is the smallest closed set $\cX$ such that $\mu(\R^d\setminus \cX)=0$.

We consider the following very general question. We remark that the below  set $\cX$ can be some fractal set.

\begin{question}  
\label{quest:ABC}
Let $\mu$ be a Radon measure on $\T_d$ with $\spt \mu =\cX$. What can we say about   
\begin{align*}
&A.~\text{mean value bounds:} \quad \int_{\cX} |S_{\va, d}(\vx; N)|^{\rho} d\mu(\vx);\\
&B.~\text{typical bounds:} \quad  \sup \left\{\alpha\in [0,1]:~ \mu(\cE_{\va, d, \alpha}\cap \cX)=0 \right\};\\
&C.~\text{exceptional sets:} \quad  \sup \left\{\alpha\in [0,1]:~\dim (\cE_{\va, d,\alpha}\cap \cX)=\dim \cX\right\}; 
\end{align*} 
provided the measure $\mu$ has some natural geometric, algebraic or combinatorial structure? 
\end{question}

For example, the restriction results from~\cite{HH1, HH2, HuWo, LD, Wool4} address some instances 
of Question~\ref{quest:ABC}~A in the case when $\cX$ is hyperplane formed by vectors $\vx \in \Tor$
with some components fixed (often to zero). 

Furthermore, there are other types of the metric theory of Weyl sums related to Question~\ref{quest:ABC}~B. 
More precisely,  let 
$$
\{\varphi_1(T), \varphi_2(T), \ldots, \varphi_d(T)\}=\{T, T^2, \ldots, T^d\}.
$$
Note that here the order of $\varphi_1, \ldots, \varphi_d$ is not specified. 
The works of~\cite{ChSh-New, FlFo, Wool3} imply that  for almost all $(x_1, \ldots, x_k)\in \T_k$ (with respect to the $k$-dimensional Lebesgue measure)  one has 
$$
\sup_{(y_1, \ldots, y_{d-k})} \left |\sum_{n=1}^{N}\e \left (\sum_{j=1}^{k} x_j \varphi_j(n)+\sum_{j=k+1}^{d}y_j \varphi_j(n)\right ) \right |\le N^{1/2+\delta(d, k)+o(1)}
$$
as $N\rightarrow \infty$,  for some explicit values $0<\delta(d, k)<1$;  we refer to~\cite{ChSh-New} for more details and the 
currently best know results in general.   We note that recently   special forms of such  bounds, using a very different approach, 
have been given in~\cite{BPPSV, ES}, together with applications
to some partial differential equations.

Here we are interested in general spaces $\cX$ and measures $\mu$ and also in 
some special cases such as {\it spheres\/}~\eqref{eq:Sphere}, 
{\it moment curves\/}~\eqref{eq:moment-curve}  and 
{\it line segments\/}~\eqref{eq:segment}.  

\subsection{Notation and conventions}
Throughout the paper, the notation $U = O(V)$, 
$U \ll V$ and $ V\gg U$  are equivalent to $|U|\leqslant c V$ for some positive constant $c$, 
which throughout the paper may depend on the degree $d$ and occasionally on the small real positive 
parameter $\varepsilon$.   

For any quantity $V> 1$ we write $U = V^{o(1)}$ (as $V \to \infty$) to indicate a function of $V$ which 
satisfies $|U| \le V^\eps$ for any $\eps> 0$, provided $V$ is large enough. One additional advantage 
of using $V^{o(1)}$ is that it absorbs $\log V$ and other similar quantities without changing  the whole 
expression.

 We use $\# \cS$ to denote the cardinality of a finite set $\cS$.

 We always identify $\Tor$ with half-open unit cube $[0, 1)^d$, in particular we
 naturally associate the Euclidean norm  $\|\vx\|$ with points $\vx \in \Tor$. 
 
 We always suppose that $d\ge 2$. 
 
For a measure  $\mu$ on $\cX$ we say that some property holds for {\it $\mu$-almost all $\vx\in \cX$\/}  if it holds 
 for a set $\cA\subseteq \cX$ such that $\mu(\cX\setminus\cA)=0$.
 
For each $q>0$ denote  
\begin{equation}
\label{eq:sq}
s(q)=q(q+1)/2.
\end{equation}

\section{Main results}

\subsection{General sets}

We consider Radon measure $\mu$ on $\T_d$ which implies that $\mu$ is a Borel measure and $\mu(\T_d)<\infty$, 
see~\cite[Chapter~1]{EG} for the general measure theory.

The Fourier transform of a Radon measure  $\mu$ on $\R^d$ is defined as 
$$
\widehat{\mu}(\bxi)=\int_{\R^d} \e(-\vx \cdot\bxi) d\mu(\vx), \quad \bxi\in \R^d,
$$
where, as usual, the dot product $\vx \cdot\bxi$  of vectors $\vx=(x_1,\ldots, x_d)$ and   $\bxi=(\xi_1, \ldots, \xi_d)$ is given by
$$
\vx\cdot \bxi= x_1\xi_1+\ldots +x_d \xi_d, 
$$
see~\cite[Chapter~3]{Mattila2015} for  the basic properties of the Fourier transform of measures.

We consider classes of  Radon measures  $\mu$  on $\Tor$ such that 
\begin{equation}
\label{eq:decay}
\left |\widehat{\mu}(\bxi)\right| \ll \(1+\|\bxi\|\)^{-\sigma}, \quad \forall \bxi \in \R^{d},
\end{equation}
for some $\sigma> 0$.

We are mostly interested in sequences $\va=(a_n)_{n=1}^{\infty}$  of complex weights such that 
\begin{equation}
\label{eq:single a}
a_n= n^{o(1)}, \qquad n\rightarrow \infty.
\end{equation}

\begin{theorem}
\label{thm:MVT}
Let  $\mu$ be a Radon measure on $\Tor$   such that~\eqref{eq:decay} holds
for some $\sigma \ge 1/d$, and let  $\va=(a_n)_{n=1}^{\infty}$ satisfy condition~\eqref{eq:single a}. Then      
$$
\int_{\Tor}\left|S_{\va, d}(\vx; N)\right| ^{2} d\mu(\vx)\le N^{1+o(1)}.
$$
\end{theorem}

For the case $\sigma>1/d$ the bound in Theorem~\ref{thm:MVT}
 is essentially optimal, see Remark~\ref{rem:optimal} below.  Moreover, it  is interesting to know whether Theorem~\ref{thm:MVT} still holds under the weaker condition that $\sigma>0$.

We remark that there are many measures  which satisfy the condition~\eqref{eq:decay}. These include  some  {\it surface measures\/} (for example, of spheres and paraboloids), see~\cite[Section~14.3]{Mattila2015}; some {\it fractal measures\/} (for example, 
natural measures on the trajectories of Bownian motion, see~\cite[Chapter~12]{Mattila2015} and some random Cantor measures, see~\cite{B}). Thus  Theorem~\ref{thm:MVT} claims these measures admit the  square mean value theorems. 

For the higher order mean value bounds we have Theorem~\ref{thm:MVT-higher-box} and~\ref{thm:MVT-higher} below, which depend
on the rate of decay of Fourier coefficients   and on boundedness of their $\mathsf L^1$-norm, respectively.

For $x\in \R$ we define  $\cl{x}$ as the nearest integer of the number $x$ if $x-1/2 \not \in \Z$ and 
also set  $\cl{x} = x+1/2$  if   $x-1/2 \in \Z$.

\begin{theorem}  
\label{thm:MVT-higher-box} 
Let  $\mu$ be a Radon measure on $\Tor$ such that~\eqref{eq:decay} holds  for some positive $\sigma\le d$
 and let  $\va=(a_n)_{n=1}^{\infty}$ satisfy the condition~\eqref{eq:single a}. Then      
$$
\int_{\Tor}\left|S_{\va, d}(\vx; N)\right|^{2s(d)} d\mu(\vx)\le N^{s(d) + s(\ell) +\ell(d-\sigma-\ell)+o(1)},
$$
where $\ell =\cl{d-\sigma+1/2}$ and $s(\ell)$ is given by~\eqref{eq:sq}. 
\end{theorem}

For the natural   measure  $\muS$ on the $d$-dimensional sphere 
\begin{equation}
\begin{split} 
\label{eq:Sphere} 
\Sp^{d-1} =\{\vt & = (t_1, \ldots, t_d)\in \R^d:~\\
&\quad (t_1-1/2)^2+\ldots+(t_d-1/2)^2 = 1/4\}  
\end{split}
\end{equation}  
centred at $(1/2, \ldots, 1/2)$ and of radius $1/2$,
we can take $\sigma = (d-1)/2$ in~\eqref{eq:decay}, see~\cite[Equation~(3.42)]{Mattila2015}.
That is, we have
\begin{equation}
\label{eq:muS}
\widehat{\muS}(\bxi)\ll (1+\|\bxi\|)^{-(d-1)/2}, 
\end{equation}
for any $ \bxi\in \R^d$. 
Substituting in Theorem~\ref{thm:MVT-higher-box} we see the following.

\begin{example} 
\label{exam:MVT S M}  
Let $a_n=n^{o(1)}$. Then for  the sphere  we have  
$$
\ell=\cl{d - (d-1)/2 + 1/2} = \cl{d/2  + 1}.
$$
Hence 
\begin{itemize}
 \item[(i)] if $d$ is even then  $\ell=    d/2+1$, thus
 \begin{align*}
  \int_{\Sp^{d-1}}   \left|\sum_{n=1}^{N} a_n\e(t_1n+\ldots +t_d n^d)\right|^{2s(d)}& d\muS(\vt)\\
&  \le N^{s(d)+(d+2)^2/8+o(1)};
\end{align*}   

\item[(ii)] if $d$ is odd then   $\ell=   d+3/2$, thus
\begin{align*}
  \int_{\Sp^{d-1}}   \left|\sum_{n=1}^{N} a_n\e(t_1n+\ldots +t_d n^d)\right|^{2s(d)} &d\muS(\vt)\\
&    \le N^{s(d)+(d+3)(d+1)/8+o(1)}. 
\end{align*}   
 \end{itemize} 
 \end{example}

We remark that  for
the natural   measure  $\muM$ on the moment curve 
\begin{equation}
\label{eq:moment-curve} 
\Gamma=\{(t, \ldots, t^{d}):~t\in [0,1]\}, 
\end{equation} 
 that is,  
\begin{equation}
\label{eq:moment-measure}
\muM \(\{(t, \ldots, t^{d}):~t\in [a,b]\}\)=b-a,
\end{equation}
by Lemma~\ref{lem:moment-decay} below we can take $\sigma = 1/d$  in~\eqref{eq:decay}, that is,
\begin{equation}
\label{eq:muM}
 \widehat{\muM}(\bxi)\ll  \(1+\|\bxi\|\)^{-1/d},
\end{equation}
for any $ \bxi\in \R^d$.  Therefore, for all $d\ge 2$ we have $\ell=\cl{d-1/d+1/2} = d$, thus Theorem~\ref{thm:MVT-higher-box} implies 
\begin{equation} 
\label{eq:MVT-Mon-Triv}
 \int_0^1  \left|\sum_{n=1}^{N} \e(tn+\ldots +t^d n^d)\right|^{2s(d)} dt
\le N^{2s(d) -1+o(1)},
\end{equation}   
which is   the same bound as one can instantly derive from Theorem~\ref{thm:MVT}. 
In fact one cannot improve the bound~\eqref{eq:MVT-Mon-Triv} as it is easy to  
see that for $0\le t \le 0.1 N$ we have  
$$
\sum_{n=1}^{N} \e(tn+\ldots +t^d n^d) \gg N
$$
and hence   
\begin{align*}
 \int_0^1&   \left|\sum_{n=1}^{N} \e(tn+\ldots +t^d n^d)\right|^{2s(d)} dt \\
 & \qquad \ge 
 \int_0^{0.1 N}  \left|\sum_{n=1}^{N} \e(tn+\ldots +t^d n^d)\right|^{2s(d)} dt\gg N^{2s(d) -1}.
\end{align*}
However in  Theorem~\ref{thm:MVT-Mom} below we use a different argument and  obtain a much 
stronger bound on the modification of  the above integral in~\eqref{eq:MVT-Mon-Triv} over the interval $[\delta, 1]$ for any positive $\delta$. Thus  for this case, we improve the bound in Theorem~\ref{thm:MVT-higher-box}.

Next, we show that   the proof of Theorem~\ref{thm:MVT-higher-box} implies the following  result.

\begin{theorem}
\label{thm:MVT-higher}
Let  $\mu$ be a Radon measure on $\Tor$ such that
\begin{equation}
\label{eq:L1}
\sum_{\xi\in \Z^d}\left |\widehat{\mu}(\bxi)\right| \ll 1,
\end{equation}
and  let  $\va=(a_n)_{n=1}^{\infty}$ satisfy the  condition~\eqref{eq:single a}. Then 
$$
\int_{\Tor} \left|S_{\va, d}(\vx; N)\right| ^{2s(d)} d\mu(\vx)\le N^{s(d)+o(1)}.
$$
\end{theorem}

Theorem~\ref{thm:almostall}  below claims that 
we can derive the ``almost all" individual  bounds by using mean value theorems.  For this purpose  we now need to  consider the family of sums similar to $S_{\va, d}(\vx; N)$ in~\eqref{eq:S_ad}, but in the following the 
weights change with $N$.  More precisely,  let $a(N, n)$ be a ``double sequence" of complex weights such that  
\begin{equation} 
\label{eq:double a}
\max_{n=1, \ldots, N}|a(N, n)|\le N^{o(1)}, \quad N\rightarrow \infty.
\end{equation}
We only consider the values $a(N, n)$ with  $1\le n\le N$ throughout the paper.

\begin{theorem}\label{thm:almostall}
Let  $\mu$ be a Radon measure on $\T_d$ and let $\rho>0$ and $0<\vartheta<1$ be two constants such that for any 
double sequence $a(N, n)$ with the condition~\eqref{eq:double a} one has 
$$
\int_{\Tor} \left|\sum_{n=1}^{N} a(N, n)\e(x_1n+\ldots +x_d n^d)\right|^{\rho} d\mu(\vx)\le N^{\vartheta\rho+o(1)}, 
$$
then for any complex  sequence $a_n=n^{o(1)}$ and for $\mu$-almost all $\vx\in \spt \mu$ we have 
\begin{equation}
\label{eq:S_ad-1/2}
\left|S_{\va, d}(\vx; N)\right| \le N^{\vartheta+o(1)}, \quad N\rightarrow \infty.
\end{equation}
\end{theorem}

\begin{remark}
\label{rem:almostall}
Let $\mu$ be a Radon measure on $\T_d$ with the property~\eqref{eq:decay} for some $\sigma\ge 1/d$. By using the similar arguments as in the proof of Theorem~\ref{thm:MVT}, we derive that  the measure $\mu$ satisfies the bound of Theorem~\ref{thm:almostall} 
with $\rho=2$ and $\vartheta=1/2$.  Therefore we obtain that for $\mu$-almost all $\vx\in \spt \mu$ we have the bound~\eqref{eq:S_ad-1/2}. 
\end{remark}

Applying  Theorem~\ref{thm:almostall}, Remark~\ref{rem:almostall}  and the bounds of~\eqref{eq:muS}
and~\eqref{eq:muM}, we obtain the square root cancellation in the following special cases of spheres and moment curves.

\begin{example}
\label{exam:almostall S M}
Let $\mu$ be the spherical measure $\muS$ or the natural measure  $\muM$ on the moment curve~\eqref{eq:moment-curve} and $a_n=n^{o(1)}$, then for $\mu$-almost all $\vx\in \spt \mu$ we have 
$$
|S_{\va, d}(\vx; N)|\le N^{1/2+o(1)}, \quad N\rightarrow \infty.
$$
\end{example}

Under the conditions of Theorem~\ref{thm:almostall}, we obtain that the bound~\eqref{eq:S_ad-1/2} holds for $\mu$-almost all $\vx\in \spt \mu$. Thus for any $\eps>0$ and $a_n=n^{o(1)}$ we have 
$$
\mu(\cE_{\va, d, \vartheta+\eps})=0,
$$ 
where the exceptional set $\cE_{\va, d, \vartheta+\eps}$ is given by~\eqref{eq:exceptional}.

For any $\alpha>\vartheta$ we  study the Hausdorff dimension of the  exceptional set $\cE_{\va, d, \alpha}$  of Theorem~\ref{thm:almostall}, that is the set for which~\eqref{eq:S_ad-1/2} fails.  
 For this purpose we need impose some regularity properties on the measure $\mu$.

\begin{definition}
\label{def:f-regular}
Let $\R_+=(0,  \infty)$ be the set of  positive real numbers. For $\vx\in \R^d$ and $\bzeta=(\zeta_1, \ldots, \zeta_d)\in \R_+^d$, we define the  $d$-dimensional rectangle (or box) with the centre $\vx$ and the side lengths $2\bzeta$ by 
$$
\cR(\vx, \bzeta)=[x_1-\zeta_1, x_1+\zeta_1)\times \ldots \times [x_d-\zeta_d, x_d+\zeta_d).
$$
Let $\mu$ be a Radon measure on $\R^d$. Suppose that there exists a  function $f:\R_+^d\rightarrow \R_+$  such that for any $\bzeta\in \R_+^{d}$  one has 
$$
\mu(\cR(\vx, \bzeta))\ge f(\bzeta), \qquad \forall \vx\in \spt \mu, 
$$
then we say that  the measure $\mu$ is $f$-regular.
\end{definition}

To illustrate Definition~\ref{def:f-regular}, we give  the following example of an $f$-regular measure. Let $L$ be a segment of $\R^d$ and $\mu$ be the natural measure on  $L$. Then for any $\bzeta=(\zeta_1, \ldots, \zeta_d)\in \R_+^d$ and any $\vx\in L$ we have   
$$
\mu(\cR(\vx, \bzeta))\ge \min\{\zeta_1, \ldots, \zeta_d\}. 
$$
Thus the measure $\mu$ is $f$-regular with $f(\bzeta) = \min\{\zeta_1, \ldots, \zeta_d\}$. 

\begin{theorem}
\label{thm:HD}
Let  $\mu$ be a $f$-regular  Radon measure on $\T_d$ for some function $f$ and let $\rho>0$,  $0<\vartheta<1$ be two constants such that for any 
double sequence $a(N, n)$ with the condition~\eqref{eq:double a} one has 
$$
\int_{\Tor}\left|\sum_{n=1}^{N} a(N, n)\e(x_1n+\ldots +x_d n^d)\right|^{\rho} d\mu(\vx)\le N^{\vartheta\rho+o(1)}.
$$
Then we have 
$$
\dim (\cE_{\va, d, \alpha} \cap \spt \mu)\le \mathfrak{u}_{\rho, \vartheta}(f; d, \alpha),
$$  
with  
\begin{align*}
\mathfrak{u}_{\rho, \vartheta}(f; d, \alpha)&=\inf \{t>0:~ \exists\, \varepsilon>0  \text{ such that }\\
&\qquad\quad\sum_{i=1}^\infty N_i^{\rho \vartheta -\rho \alpha+s(k)+t(\alpha-k-2)+ c(d)\varepsilon} f(\bzeta_i(\varepsilon))^{-1}<\infty,\\
& \qquad\qquad\qquad\qquad\qquad\qquad \text{ for some } k=0, \ldots, d-1 \},
\end{align*}  
where $c(d)$ is a positive constant which depends  only on $d$ and  
$$
\bzeta_{i}(\varepsilon)=\(\zeta_{i,1}(\varepsilon), \ldots, \zeta_{i, d}(\varepsilon)\)
$$
with  $\zeta_{i, j}(\varepsilon)=N_i^{\alpha-j-1-\varepsilon}$, $ j=1\ldots, d$, and $N_i=2^i$, $i\in \N$.
\end{theorem}

\begin{remark} 
As we have claimed before there are many measures that satisfy the condition~\eqref{eq:decay}, 
and thus fulfil the conditions in Theorems~\ref{thm:MVT} and~\ref{thm:MVT-higher-box}. Moreover, applying the similar arguments as in the proofs of Theorems~\ref{thm:MVT} and~\ref{thm:MVT-higher-box}, one obtains that their conclusions 
 still hold even when we take any  double sequence $a(N, n)$ with the condition~\eqref{eq:double a} instead
  of the sequence $a_n$. Thus these measures also satisfy the  mean value bounds  in Theorem~\ref{thm:HD}. Furthermore, many measures  also satisfy the $f$-regular condition of Definition~\ref{def:f-regular} for some function $f$. Thus for these measures we deduce the dimension bounds for the sets $\cE_{\va, d, \alpha} \cap \spt \mu$. Below we give some concrete examples of applications of Theorem~\ref{thm:HD}.  
\end{remark}

We remark that if $\mu$ is the Lebesgue measure on $\T_d$ then for any rectangle $\cR(\vx, \bzeta)$ one has  
$$
\mu(\cR(\vx, \bzeta))= f(\bzeta) \quad \text{with}\quad  f(\bzeta) = \prod_{j=1}^{d}\zeta_j.
$$ 
For this special case and for  $\dim \cE_{\va, d, \alpha}$,
after simple calculations we obtain the same upper bound 
as~\eqref{eq:upper-u} for $\cE_{d, \alpha}$. More precisely, we have the following.
 
\begin{example}
Let $\mu$ be the Lebesgue measure on $\T_d$, then 
for $1/2<\alpha<1$ and $a_n=n^{o(1)}$ we have 
$\dim\cE_{\va, d, \alpha}\le \mathfrak{u}(d, \alpha)$,
where $\mathfrak{u}(d, \alpha)$ is given by~\eqref{eq:upper-u}. 
\end{example} 

Now we consider the Weyl sums on sphere $\Sp^{d-1}$.  Observe that for  any $\bzeta\in \R_+^{d}$ with 
$$
1>\zeta_1\ge \ldots\ge \zeta_d, 
$$
we have  
$$
\muS(\cR(\vx, \bzeta)\cap \Sp^{d-1})\gg  \prod_{j=2}^{d}\zeta_j, \qquad \vx\in \Sp^{d-1}, 
$$
hence $\muS$ is $f$-regular if we take  
\begin{equation}\label{eq:sphere+box}
f(\bzeta)  = c_0(d)\zeta_2 \ldots \zeta_d, 
\end{equation}
for some constant $c_0(d)>0$ 
which depends only on $d$.

Furthermore suppose that $\muS$ satisfies the condition in Theorem~\ref{thm:HD} for some $\rho$ and $\vartheta$,  then using~\eqref{eq:sphere+box}  in the setting of  Theorem~\ref{thm:HD},
we see that 
$$
f(\bzeta_i(\eps))\gg N_i^{(d-1)(\alpha-1-\eps)-s(d)+1},
$$
and we obtain that
$$
\dim (\cE_{\va, d, \alpha}\cap \Sp^{d-1})\le t
$$
provided that there exists $k=0, 1, \ldots, d-1$ such that 
$$
\rho\vartheta -\rho \alpha+s(k)+t(\alpha-k-2)+(1-\alpha)(d-1)+s(d)-1<0,
$$
which is the same as 
$$
t>\frac{\rho \vartheta-\rho \alpha+s(d)+s(k)+(d-1)(1-\alpha)-1}{k+2-\alpha}.
$$

We formulate these arguments into the following.

\begin{example}
\label{ex:sphere-dim}  
Suppose that there exists $\rho>0$, $0<\vartheta<1$ such that for any double sequence $a(N, n)$ with the condition~\eqref{eq:double a} one has 
$$
\int_{\Tor}\left|\sum_{n=1}^{N} a(N, n)\e(x_1n+\ldots +x_d n^d)\right|^{\rho} d\muS(\vx)\le N^{\vartheta\rho+o(1)}.
$$ 
Then we have 
$$
\dim (\cE_{\va, d, \alpha} \cap \Sp^{d-1})\le \mathfrak{u}_{\mathsf {S}, \rho, \vartheta}(d, \alpha),
$$ 
where
$$
\mathfrak{u}_{\mathsf {S}, \rho, \vartheta}(d, \alpha)=\min_{k=0, 1, \ldots, d-1}\frac{\rho \vartheta-\rho \alpha+s(d)+s(k)+(d-1)(1-\alpha)-1}{k+2-\alpha}.
$$ 
\end{example}

\begin{remark}
It is expected that  for some specific sets we 
could obtain better bounds than the general one of Theorem~\ref{thm:HD}. Below we give such results 
for  moment curves~\eqref{eq:moment-curve}  and line segments~\eqref{eq:segment}. 
\end{remark} 

\subsection{Moment curves}  
We start with giving an improved version  of the bound~\eqref{eq:MVT-Mon-Triv} when we integrate 
over the interval $[\delta, 1]$.

\begin{theorem} 
\label{thm:MVT-Mom} 
For any  $s\ge 1$ and  any $0<\delta<1$ and $N >1/\delta$ we have 
\begin{align*}
\int_\delta^{1} &
   \left|\sum_{n=1}^{N} a_n\e(tn+\ldots +t^d n^d)\right|^{2s} dt\\
& \qquad \qquad \le  \delta^{(1-d)/2} \(N^{s} +  N^{2s - s(d)/2}\)N^{s(d)/2-d/2+o(1)}.
\end{align*}
\end{theorem}

Note that for a fixed $\delta$, by Theorem~\ref{thm:MVT} and Lemma~\ref{lem:moment-decay} we obtain 
$$
\int_\delta^{1} 
   \left|\sum_{n=1}^{N} a_n\e(tn+\ldots +t^d n^d)\right|^{2} dt\le N^{1+o(1)},
$$
and hence for any $s\ge 1$,
\begin{equation}
\label{eq:triv}
\int_\delta^{1} 
   \left|\sum_{n=1}^{N} a_n\e(tn+\ldots +t^d n^d)\right|^{2s} dt\le N^{2s-1+o(1)}.
\end{equation}
Clearly, for $s \ge s(d)/2$ the bound of Theorem~\ref{thm:MVT-Mom} takes form 
\begin{equation}
\label{eq:2s(d)}
 \int_\delta^{1}  \left|\sum_{n=1}^{N} a_n\e(tn+\ldots +t^d n^d)\right|^{2s} dt \le   \delta^{(1-d)/2} N^{2s-d/2+o(1)}.
\end{equation}
Thus for a fixed $\delta$, the bound~\eqref{eq:2s(d)} improve the trivial  bound~\eqref{eq:triv} for any $d>2$.

We remark that one can obtain a similar result for the integral over any interval $[\delta, L]$ from some $0<\delta<L$.

It is  interesting to understand whether the exponent $2s-d/2$  is optimal in~\eqref{eq:2s(d)}. However we have the following lower bound. Note that there exits a  small constant  $\varepsilon_0>0$ such that $t\in [1-\varepsilon_0/N^{d}, 1]$  implies 
$$
\left|\sum_{n=1}^{N} \e(tn+\ldots +t^d n^d) \right | \ge \frac{1}{2} N,
$$
and hence, say for $\delta < 1/2$ and any $s>0$,
\begin{align*}
\int_\delta^{1} &
   \left|\sum_{n=1}^{N} \e(tn+\ldots +t^d n^d)\right|^{2s} dt\\
 & \qquad \ge  \int_{1-\varepsilon_0/N^{d}}^{1}    \left|\sum_{n=1}^{N} \e(tn+\ldots +t^d n^d)\right|^{2s} dt \\
 & \qquad \ge 
 \int_{1-\varepsilon_0/N^{d}}^{1}  \left|\sum_{n=1}^{N} \e(tn+\ldots +t^d n^d)\right|^{2s} dt
\gg N^{2s -d}.
\end{align*}

For the moment curve $\Gamma$ defined by~\eqref{eq:moment-curve}, 
Example~\ref{exam:almostall S M} asserts that for 
$\muM$-almost all $\vx\in \Gamma$  
one has  
$$
|S_{\va, d}(\vx; N)|\le N^{1/2+o(1)}, \quad N\rightarrow \infty.
$$

For the exceptional sets $\cE_{\va, d, \alpha}\cap \Gamma$ we have the following.

\begin{theorem}
\label{thm:moment}
For the moment curve $\Gamma$ defined by~\eqref{eq:moment-curve} and any $a_n=n^{o(1)}$ we have 
$$
\dim (\cE_{\va, d, \alpha}\cap \Gamma) \le 1-\frac{2\alpha-1}{d+1-\alpha}.
$$
\end{theorem}

We note that Theorem~\ref{thm:moment} gives a non-trivial bound for any $0<\alpha<1$. 
If $\alpha\rightarrow 1$ then the bounds of~\eqref{eq:upper-u} gives   
$$
\dim (\cE_{\va, d, \alpha} \cap \Gamma) \le \dim \cE_{\va, d, \alpha} \rightarrow 0,
$$
which is better than  the bound in Theorem~\ref{thm:moment}. 
However, if $\alpha$ is close to $1/2$ then Theorem~\ref{thm:moment} implies a  better bound.

It is natural to expect   that if $\alpha \rightarrow 1/2$ then
$$
\dim \cE_{\va, d, \alpha}\rightarrow d \quad \text{and} \quad \dim (\cE_{\va, d, \alpha}\cap \Gamma) \rightarrow 1.
$$

\subsection{Segments}

We investigate the Weyl sums on the given segment. Among other things, our results imply that  the  condition~\eqref{eq:decay} of Theorem~\ref{thm:MVT}  is not necessary. 

We introduce some notation first. Let $\bm{\omega} \in \R^d$ with $\|\bm{\omega}\|=1$ and 
\begin{equation}
\label{eq:segment}
L_{\bm{\omega}}=\{t\bm{\omega}:~ t\in [0,1]\}.
\end{equation} 
Let $\muo$ be the Lebesgue measure on $L_{\bm{\omega}}$. The orthogonal complementary space of $\bomega$ is defined as 
$$
\bomega^{\perp}=\{\vx\in \R^d:~\vx\cdot \bomega=0\}.
$$
Clearly for any $\bxi\in \bomega^{\perp}$ we have  
$$
\widehat{\muo}(\bxi)=\int_{L_{\bomega}} \e(-\vx\cdot \bxi) d\muo (\vx)=\int_{0}^{1}\e(-t\, \bxi\cdot \bomega) dt=1.
$$
Thus the measure $\muo$ does not have the decay property as~\eqref{eq:decay}. However, by using the van der Corput  lemma
(see~\cite[Theorem~14.2]{Mattila2015} or Lemma~\ref{lem:var} below)   we obtain an analogue of the result  of Theorem~\ref{thm:MVT}.

\begin{theorem}
\label{thm:MVT-L}
Using the above notation, let $\va=(a_n)_{n=1}^{\infty}$ satisfy  the condition~\eqref{eq:single a}, then 
$$
\int_{L_{\bomega}}\left|S_{\va, d}(\vx; N)\right| ^{2} d\muo(\vx)\le N^{1+o(1)}.
$$
\end{theorem}

\begin{cor}
\label{cor:almostall-segment}
Using the above notation, for any sequence $a_n=n^{o(1)}$ and for $\muo$-almost all $\vx\in L_{\bomega}$ one has
$$
|S_{\va, d}(\vx; N)|\le N^{1/2+o(1)}.
$$ 
\end{cor}

\begin{cor}
\label{cor:HD-segment}
Let $a_n=n^{o(1)}$ and  $1/2<\alpha<1$ then for any 
\begin{equation}
\label{eq:omega-k}
\bomega=(\omega_1, \ldots, \omega_k, 0, \ldots, 0)\in \R^d, \quad \omega_k\neq 0,\  \|\bomega\|=1,
\end{equation}
we have 
$$
\dim (\cE_{\va, d, \alpha} \cap L_{\bomega})\le 1-\frac{2\alpha-1}{k+1-\alpha}.
$$
\end{cor}

Note that the condition~\eqref{eq:omega-k} is used in the following way. For any $\vx\in L_{\bomega}$ and $\bzeta=(\zeta_1, \ldots, \zeta_d)$ with $0<\zeta_j<1$, $j=1, \ldots, d$, we have 
$$
\zeta_k\ll \diam(L_{\bomega} \cap \cR(\vx, \bzeta))\ll \zeta_k.
$$

We remark that  if $\alpha\rightarrow 1$ then the bounds of~\eqref{eq:upper-u} gives 
$$
\dim (\cE_{\va, d, \alpha} \cap \spt \muo) \le \dim \cE_{\va, d, \alpha} \rightarrow 0,
$$
which is better than the bound in  Corollary~\ref{cor:HD-segment}.  However the following Example~\ref{exam:Hor Segm} shows that Corollary~\ref{cor:HD-segment} can give better bounds in some cases. 

\begin{example}
\label{exam:Hor Segm} 
For the horizontal segment 
$$
L=\{(t, 0):~t\in [0,1)\}\subseteq \T_2,
$$
and for  any $1/2<\alpha<1$,  Corollary~\ref{cor:HD-segment} with $d=2$ and $k=1$ implies that 
$$
\dim (\cE_{d, \alpha} \cap L)\le \frac{3(1-\alpha)}{2-\alpha}.
$$
\end{example}

While applying~\eqref{eq:upper-u} and~\eqref{eq:def u}  and using $k=0, 1$  we obtain
$$
\dim (\cE_{d, \alpha} \cap L)\le  \dim \cE_{d, \alpha} \le \min\left \{\frac{8(1-\alpha)}{2-\alpha}, \frac{9-8\alpha}{3-\alpha}\right \}.
$$
For $1/2<\alpha<1$ elementary calculus    shows that 
$$
\frac{3(1-\alpha)}{2-\alpha}<\min\left \{\frac{8(1-\alpha)}{2-\alpha}, \frac{9-8\alpha}{3-\alpha}\right \}.
$$
Therefore,  the bound of Example~\ref{exam:Hor Segm}
gives a better bound for all $1/2<\alpha<1$. 
 
In general,  the exact comparison between the bound $\mathfrak{u}(d, \alpha)$ and that of Corollary~\ref{cor:HD-segment} is not immediately obvious.

\section{Preliminaries}

\subsection{The completion technique} 
We remark that the completion technique has many applications in analytic number theory.  The following  bound is a special case of~\cite[Lemma~3.2]{ChSh-New}. 
\begin{lemma}\label{lem:completion} 
For $\vx\in \Tor$ and $1\le M\le N$ we have 
$$
S_{\va, d}( \vx; M) \ll W_{\va, d}( \vx; N),
$$
where 
$$
W_{\va, d}( \vx; N)=  \sum_{h=-N}^{N} \frac{1}{|h|+1} \left| \sum_{n=1}^{N}  a_n \e\(h n/N+x_1n+\ldots+x_dn^d \) \right|.
$$
\end{lemma}

Note that for any $N$ there exists a sequence $b_N(n)$ such that 
$$
b_N(n) \ll \log N, \qquad n=1, \ldots, N, 
$$ 
and $W_{\va, d}( \vx; N)$ can be written as 
$$
W_{\va, d}( \vu; N)=\sum_{n=1}^N a_n b_N(n)  \e(x_1 n+\ldots + x_d n^d).
$$ 
Indeed, for each $h\in \Z$, $N\in \N$, $\vx\in \T_d$ and the sequence $\va=(a_n)_{n=1}^{\infty}$ there exists some complex  number $\vartheta(h, N, \vx, \va)$ on the unit circle such that 
$$
W_{\va, d}( \vx; N)= \sum_{h=-N}^{N} \frac{\vartheta(h, N, \vx, \va)}{|h|+1}   \sum_{n=1}^{N} a_n  \e\(h n/N+x_1 n+\ldots + x_d n^d \).
$$                           
 Hence 
$$
b_N(n) =  \sum_{h=-N}^{N} \frac{\vartheta(h, N, \vx, \va)}{|h|+1}  \e(h n/N) \ll \log N. 
$$ 

\subsection{Continuity of exponential  sums}
Analogously to~\cite[Lemma~3.5]{ChSh-New}  and~\cite[Lemma~2.1]{Wool3} we obtain:

\begin{lemma} 
\label{lem:cont} 
Let $0<\alpha<1$ and let $\varepsilon>0$ be  sufficiently small. If 
$|W_{\va, d}(\vx; N)|\ge N^{\alpha}$ for  some $\vx\in \Tor$,  then 
$$
|W_{\va, d}( \vy; N)|\ge  N^{\alpha}/2
$$
holds for  any $\vy\in \cR(\vx, \bzeta)$ provided  that $N$ is large enough and 
$$
0<\zeta_j\le N^{\alpha-j-1-\eps},  \qquad j =1, \ldots, d.
$$
\end{lemma}
\begin{proof} For any integer $h$ with $|h|\le N$ we have  
\begin{align*}
\sum_{n=1}^{N}   a_n\e(h n/N) & \left(\e\(x_1n+\ldots+x_dn^d \) - \e\(y_1n+\ldots+y_dn^d \) \right)\\
& \quad\quad\quad\quad\quad\quad \ll \sum_{n=1}^{N} \sum_{j=1}^d  a_n\zeta_j n^j\le N^{\alpha-\eps/2}.
\end{align*}
The last estimate holds for all sufficiently large   $N$. By Lemma~\ref{lem:completion} we obtain 
$$
|W_{\va, d}( \vx; N)-W_{\va, d}(\vy; N)| \ll N^{\alpha-\eps/2} \log N \le N^{\alpha}/2,
$$
which holds for all  sufficiently large  $N$ and thus the result follows. 
\end{proof}

\subsection{Covering the large values of exponential  sums}

In analogy to~\cite[Lemma~3.7]{ChSh-New} we obtain the following result. 

\begin{lemma}
\label{lem:3r-cover}  
Let $\mu$ be a $f$-regular Radon measure  as in Theorem~\ref{thm:HD}.  
Let $0<\alpha<1$ and $\eps>0$ be a small parameter and let
$$
\bzeta(\varepsilon) = \(\zeta_1(\varepsilon), \ldots, \zeta_d(\varepsilon)\)
$$ 
where 
$$
\zeta_j(\varepsilon)=N^{\alpha-j-1-\varepsilon}, \qquad  j=1, \ldots, d.  
$$
Then for some   
$$
L\le N^{\rho(\vartheta-\alpha)+o(1)}f\(\bzeta(\varepsilon)\)^{-1},
$$ 
there exist $\vx_1, \ldots, \vx_L \in \spt \mu$ such that 
$$
\{\vx\in \spt \,\mu:~|W_{\va, d}(\vx; N)|\ge N^{\alpha}\} \subseteq \bigcup_{\ell=1}^{L} \cR\(\vx_\ell, 3\bzeta(\varepsilon)\),
$$
and 
$$
\cR\(\vx_i, \bzeta(\varepsilon)\) \cap \cR\(\vx_j, \bzeta(\varepsilon)\)=\emptyset, \qquad  1\le i\neq j\le L.
$$
\end{lemma}

\begin{proof}
Let 
$\cB=\{\vx\in \spt \,\mu:~|W_{\va, d}(\vx; N)|\ge N^{\alpha}\}$,
and let 
$$
\cR_\ell = \cR\(\vx_\ell, \bzeta(\varepsilon)\), \qquad \ell=1, \ldots, L,
$$ 
be a maximal collection of pair-wise disjoint rectangles from
the set $\{\cR\(\vx, \bzeta(\varepsilon)\):~\vx\in \cB\}$. Then,  we observe  
that  the maximality of the family  of the rectangles $\cR_\ell$, $ \ell=1, \ldots, L$, 
implies that  each blow-up rectangle $\cR\(\vx_\ell, 3\bzeta(\varepsilon)\)$ 
implies that 
$$
\cB\subseteq \bigcup_{\vx\in \cB} \cR\(\vx, \bzeta(\varepsilon)\)\subseteq \bigcup_{\ell=1}^{L} \cR\(\vx_\ell, 3\bzeta(\varepsilon)\).
$$

For each $\ell=1, \ldots, L$, applying Lemma~\ref{lem:cont} we have $$
|W_{\va, d}(\vx; N)|\ge N^{\alpha}/2, \quad \forall \vx \in \cR\(\vx_\ell, \bzeta(\varepsilon)\).
$$
Thus, together with the assumption that the measure $\mu$ is a $f$-regular measure on $\T_d$,  we arrive at
\begin{align*}
\int_{\Tor}|W_{\va, d}(\vx; N)|^\rho d \mu(\vx)&\ge \sum_{\ell=1}^{L} \int_{\T_d \,\cap\, \cR_\ell} |W_{\va, d}(\vx; N)|^\rho d\mu(\vx)\\
&\gg  \sum_{\ell=1}^{L} N^{\rho \alpha} \mu(\T_d\cap \cR_\ell)\\
&\gg L N^{\rho\alpha} f\(\bzeta(\varepsilon)\).
\end{align*} 
Therefore,  combining with the mean value bound of Theorem~\ref{thm:HD}, we obtain the desired bound for $L$.
\end{proof}

From  Lemmas~\ref{lem:completion} and~\ref{lem:3r-cover},   we formulate  the following Corollary~\ref{cor:cover}  for the convenience of our applications on estimating the Hausdorff  dimension of the set $\cE_{\va, d, \alpha}\cap \spt \mu$.   Using the same notation
as in Lemma~\ref{lem:3r-cover} we denote
$$
\fR(N)=\{\cR\(\vx_\ell, 3\bzeta(\varepsilon)\):~ \ell=1, \ldots, L\}.
$$ 

\begin{cor}    
\label{cor:cover}
Let $\mu$ be a $f$-regular Radon measure  as in Theorem~\ref{thm:HD}.  
Let $0<\alpha<1$ and let $\eps>0$ be a small parameter. Let $N_i=2^{i}$, $i\in \N$. Then for any $\eta>0$  we have 
$$
\cE_{\va, d, \alpha+\eta} \cap \spt \mu \subseteq \bigcap_{q=1}^{\infty}\bigcup_{i=q}^{\infty} \bigcup_{\cR\in \fR(N_i)} \cR, 
$$
where each $\cR$ of $\fR(N_i)$ has the side length 
 $\bzeta_{i}(\varepsilon)=(\zeta_{i,1}(\varepsilon), \ldots, \zeta_{i, d}(\varepsilon))$ such that 
$$
\zeta_{i, j}(\varepsilon)=N_i^{\alpha-j-1-\eps}, \qquad j=1, \ldots, d,
$$
and furthermore 
$$
\#  \fR(N_i)\le N_{i}^{\rho \vartheta-\rho \alpha} f\(\bzeta_i(\varepsilon)\)^{-1}.
$$
\end{cor} 

\begin{proof} 
We continue to use the same notation as in Lemmas~\ref{lem:completion} and~\ref{lem:3r-cover}. For each $i\in \N$ and $N_i=2^{i}$
let 
$$
\cB_i=\{\vx\in \spt \,\mu:~|W_{\va, d}(\vx; N_i)|\ge N_i^{\alpha}\}.
$$

We intend to show that for any $\eta>0$  we have 
\begin{equation}
\label{eq:subset}
\cE_{\va, d, \alpha+\eta} \cap \spt \mu \subseteq \bigcap_{q=1}^{\infty}\bigcup_{i=q}^{\infty} \cB_i.
\end{equation}

Let $\vx\in  \cE_{\va, d, \alpha+\eta} \cap \spt \mu$. 
Suppose that 
$$
\vx \notin \bigcap_{q=1}^{\infty}\bigcup_{i=q}^{\infty} \cB_i.
$$
Then there exists $i_{\vx}$ such that for all $i\in \N$, $i\ge i_{\vx}$ implies 
$$
|W_{\va, d}(\vx; N_i)|\le N_i^{\alpha}.
$$
Clearly for any $N>N_{i_\vx}$ there exists $i\ge i_{\vx}$ such that 
$$
N_{i}\le N<N_{i+1}.
$$
By Lemma~\ref{lem:completion} we arrive at 
$$
|S_{\va, d}(\vx; N)|\ll |W_{\va, d}(\vx; N_{i+1})|\ll  N^{\alpha}.
$$
Thus we have a contradiction with our condition $\vx\in \cE_{\va, d, \alpha+\eta}\cap \spt \mu$. Therefore, we deduce that 
$$
\vx \in \bigcap_{q=1}^{\infty}\bigcup_{i=q}^{\infty} \cB_i,
$$
and thus we have~\eqref{eq:subset}. Combining this with Lemma~\ref{lem:3r-cover} we obtain the desired result.
\end{proof}

\subsection{Bounds on exponential integrals}  
For bounding various exponential integrals (or oscillatory integrals) we  use the following
{\it lemma of van der Corput\/}, see~\cite[Theorem~14.2]{Mattila2015}. 

\begin{lemma}
\label{lem:var}  
Let $\varphi$ be a smooth function on $\R$ with $|\varphi^{(k)}(x)|\ge Q>0$ uniformly over $x\in [a, b]$ and for some $k\in\{1, 2, \ldots\}$. Then we have 
$$
\int_{a}^{b}\e(\varphi(x)) dx \ll Q^{-1/k},
$$
provided that 
\begin{itemize}
\item[(i)] either  $k=1$ and the derivative $\varphi'$ is a monotone  function   over  interval $[a, b]$; 
\item[(ii)] or  $k\ge 2$.
\end{itemize}
\end{lemma}

We now present  the following general  bound on exponential integrals with polynomial arguments which is perhaps known already. 
We give a proof here for the completeness.   

\begin{lemma}
\label{lem:moment-decay}
For  any interval $[a, b]\subseteq [0, 1]$ we have   
$$
\int_{a}^{b} \e(t\xi_1+\ldots+t^{d}\xi_d) dt \ll  \(1+\|\bxi\|\)^{-1/d}, \quad \bxi\in \R^d.
$$
\end{lemma}
\begin{proof}
Without  loss of generality we can asume $\|\bxi\|\ge 1$ in the following. Let $1\le k_0\le d$ be such that 
\begin{equation}
\label{eq:max}
|\xi_{k_0}|=\max\{|\xi_1|, \ldots, |\xi_d|\}.
\end{equation}

We now present  a case-by-case argument which depends on the value of $k_0$ and $|\xi_{k_0}|$.

{\bf Case 1.} Suppose that $k_0=d$.  Let 
$$
\varphi(t)= t\xi_1+\ldots+t^{d}\xi_d,
$$
then by~\eqref{eq:max} we obtain
$$
|\varphi^{(d)}(t)|=d!\,\xi_d\gg \|\bxi\|, \qquad  t\in [a, b].
$$
Thus  by Lemma~\ref{lem:var} we obtain the desired bound in this case. 

{\bf Case 2.} Suppose that $1\le k_0<d$ and 
$$
|\xi_{k_0}|\ge 2 \sum_{j=k_0+1}^{d} \binom{j}{k_0}|\xi_j| b^j.
$$
It follows that  for any $t\in [a, b]$ we have 
\begin{align*}
|\varphi^{(k_0)}(t)|&=k_0! \left | \xi_{k_0}  + \sum_{j=k_0+1}^{d} \binom{j}{k_0}\xi_j t^{j-k}\right |\\
&\ge k_0! \left (|\xi_{k_0}| - \sum_{j=k_0+1}^{d} \binom{j}{k_0}|\xi_j| b^{j-k}\right )\\
&\ge k_0! |\xi_{k_0}|/2.
\end{align*}
Applying Lemma~\ref{lem:var} and~\eqref{eq:max} we obtain ($\|\bxi\|\ge 1$)
$$
\int_{a}^{b}\e(t\xi_1+\ldots+t^d\xi_d)dt \ll \|\bxi\|^{-1/k_0}\le \|\bxi\|^{-1/d},
$$ 
which concludes this case.

{\bf Case 3.}
Suppose that $1\le k_0<d$ and 
$$
|\xi_{k_0}|< 2 \sum_{j=k_0+1}^{d} \binom{j}{k_0}|\xi_j| b^j.
$$
Then there exists $ k_0+1\le k_1\le d$ such that 
$$
|\xi_{k_1}| \gg \|\bxi\|.
$$
Now applying the same arguments as in  {\bf Case~1} and {\bf Case~2} to $k_1$, then either we obtain the desired bound or there exists $k_2\ge k_1+1$ such that 
$$
|\xi_{k_2}|\gg \|\bxi\|.
$$  
Iterating this argument at most $d$ times yields the desired result.
\end{proof}

\section{Proofs of general results}

\subsection{Proof of Theorem~\ref{thm:MVT}}

For  $n, m\in \N$ let 
\begin{equation}\label{eq:norm}
\bxi_{n, m}=(n-m, \ldots, n^d-m^d).
\end{equation}  
Clearly we have 
$$
|n^d-m^d|\ll \|\bxi_{n, m}\|\ll |n^d-m^d|.
$$
Let 
$$
\cI =\int_{\Tor} \left|\sum_{n=1}^{N} a_n\e(x_1 n+\ldots +x_d n^d)\right|^2 d \mu(\vx).
$$
Expanding the square and changing the order of summation,  we have
\begin{align*}
\cI  &=\int_{\Tor}\sum_{1\le n, m\le N} a_n\overline{a_m} \e\(\sum_{i=1}^d x_i (n^i-m^i)\) d\mu(\vx)\\
&= \sum_{1\le n, m\le N} a_n\overline{a_m}\, \widehat{\mu}\(\bxi_{n, m}\).
\end{align*}

Note that $\widehat{\mu}(\v0)=\mu(\T_d)<\infty$, where the second inequality holds by the definition of Radon measure. 
For $n\neq m$ applying the decay condition~\eqref{eq:decay} and~\eqref{eq:norm} we obtain 
$$
\widehat{\mu}\(\bxi_{n, m}\)\ll (n^d-m^{d})^{-\sigma}.
$$ 
Hence, using the symmetry of $m$ and $n$ we obtain
\begin{align*}
\cI&\le N^{1+o(1)}+N^{o(1)}\sum_{1\le m \neq n\le N} \(n^d-m^d\)^{-\sigma}\\
&\le N^{1+o(1)}+N^{o(1)}\sum_{m=1}^N \sum_{k=1}^N \((m+k)^d-m^d\)^{-\sigma}\\
&\le N^{1+o(1)}+N^{o(1)}\sum_{m=1}^N \sum_{k=1}^N \(k m^{d-1}\)^{-\sigma} \\
& \le  N^{1+o(1)}+N^{o(1)}\(1 + N^{1 -(d-1)\sigma}\)  \(1 + N^{1 -\sigma}\)  \le    N^{1+o(1)} 
\end{align*}
provided $\sigma \ge 1/d$.

\begin{remark}\label{rem:optimal}  For $\sigma>1/d$ and $d\ge 2$ the proof of Theorem~\ref{thm:MVT} implies  
\begin{align*} 
\int_{\Tor} & \left|\sum_{n=1}^{N} a_n\e(x_1 n+\ldots +x_d n^d)\right|^2 d \mu(\vx) 
-N\mu(\T_d) \sum_{n=1}^{N} |a_n|^2 \\
&\qquad \qquad \ll N^{o(1)}\(N^{1-\sigma}+N^{1-(d-1)\sigma}+N^{2-d\sigma}\)=
N^{1-\kappa+o(1)}, 
\end{align*} 
where $\kappa = \min\{\sigma, d\sigma - 1\} > 0$. 
Thus for  the case $\sigma>1/d$  we obtain an asymptotic formula with a power saving for the square mean values of $S_{\va, d}(\vx; N)$. In particular, this implies that the bound in Theorem~\ref{thm:MVT} is optimal when $\sigma>1/d$.  
\end{remark}

\subsection{Proof of Theorem~\ref{thm:MVT-higher-box}}

For $\bxi=(\xi_1, \ldots, \xi_d)\in \Z^{d}$ let  
 $J_{\va, N}(\bxi)$ be the  number of solutions to the system of equations
\begin{align*}
\sum_{j=1}^{s(d)}  n_j^i &- \sum_{j=s(d)+1}^{2s(d)}  n_j^i =\xi_i,   \qquad  i=1, \ldots, d, \\
 n_j  & = 1, \ldots, N,  \qquad   j =1, \ldots, 2s(d),
\end{align*} 
counted with the weights 
\begin{equation}
\label{eq:weight}
\prod_{j=1}^{s(d)} a_{n_j}  \prod_{j=s(d)+1}^{2s(d)} \overline{a_{ n_j}} = N^{o(1)}.
\end{equation}

If $a_n=1$ for all $n\in \N$ then we simply denote 
$$
J_N(\bxi)=J_{\va, N}(\bxi). 
$$
From the orthogonality of exponential functions, it is not hard to see that 
$$
J_N(\bxi) = \int_{\Tor}  \left|\sum_{n=1}^{N} \e(x_1 n+\ldots +x_d n^d)\right|^{2 s(d)}\e(-x_1 \xi_1-\ldots -x_d\xi_d) \, d  \vx.
$$ 
Hence, 
$$
J_N(\bxi)\le J_N(\v0), \qquad \forall \bxi\in \Z^{d}.
$$

Recall that the Vinogradov mean value theorem~\eqref{eq:MVT} implies  
\begin{equation}
\label{eq:MVT-Max}
J_N(\v0)\le N^{s(d)+o(1)}, \qquad N\rightarrow \infty.
\end{equation} 

From~\eqref{eq:weight} and~\eqref{eq:MVT-Max} we obtain
\begin{equation}
\label{eq:zero-max}
J_{\va, N}(\bxi)\le N^{o(1)} J_{N}(\bxi)\le N^{s(d)+o(1)}.
\end{equation}

Note that $J_{\va, N}(\bxi)=0$ if there is $i=1, \ldots, d$ such that 
$$
|\xi_i|\ge s(d)N^{i}.
$$

For each $N$ let 
\begin{equation}
\label{eq:DN}
\cD_N=\{\bxi\in \Z^{d}:~|\xi_i|\le s(d)N^{i},\  i=1,\ldots, d\}.
\end{equation}

Expanding sums $|S_{\va, d}(\vx; N)|^{2s(d)} $ and changing the order of summation, and combining with~\eqref{eq:zero-max} and the definition of $\cD_N$, for 
$$
\fI =\int_{\Tor}   \left|\sum_{n=1}^{N} a_n\e(x_1 n+\ldots +x_d n^d)\right|^{2 s(d)} d \mu(\vx)
$$  
we derive 
\begin{equation} \label{eq:A}
\fI \le  \sum_{\bxi\in \Z^{d}} J_{\va, N}(\bxi)\left |\widehat{\mu}(\bxi)\right|\le N^{s(d)+o(1)} \varXi_N,
\end{equation} 
where
$$
\varXi_N = \sum_{\bxi\in \cD_N}(1+\|\bxi\|)^{-\sigma}.
$$

Taking  dyadic decomposition  for the range $1\le \|\bxi\|  \ll N^d$ we derive that there exists a positive 
integer  $H < 2d N^d$ 
 such that 
\begin{equation}
\label{eq:XiN}
\varXi_N \ll H^{-\sigma} \#\cL_{N, H}  \log N,
\end{equation} 
where 
$$
\cL_{N, H}=\{\bxi\in \cD_N: H\le \|\bxi\|<2H\}.
$$
Assume that 
\begin{equation}
\label{eq:H}
N^i\le H <N^{i+1}
\end{equation}
for some $i=0, 1, \ldots,  d$.

From~\eqref{eq:DN} and~\eqref{eq:H}  we obtain 
$$
\#\cL_{N, H}\ll N^{s(i)}H^{d-i},
$$
where $s(i)$ is given by~\eqref{eq:sq}. 
By~\eqref{eq:XiN} we arrive at 
\begin{equation}
\label{eq:here}
\varXi_N\ll N^{s(i)} H^{d-\sigma-i}  \log N.
\end{equation}

We now formulate a case-to-case argument which depends on the value $i$ and $\sigma$.

{\bf Case 1.} Suppose that $i\le d-\sigma$. From~\eqref{eq:H} and~\eqref{eq:here} we obtain 
$$
\varXi_N\ll N^{s(i)+(i+1)(d-\sigma-i)}.
$$
Let $f(t)=t(t+1)/2+(t+1)(d-\sigma-t)$ then $f$ attains the maximal value at  
$$
t_0=d-\sigma-1/2.
$$  
and since $f$ is quadratic its largest  value at an integer argument is attained at 
\begin{equation}
\label{eq:i0}
i_0=  \cl{t_0} = \cl{d-\sigma-1/2}.
\end{equation}  
It follows that for each $i=0, 1, \ldots, d$ we have 
\begin{equation}
\label{eq:case1}
\varXi_N\ll N^{s(i)+(i+1)(d-\sigma-i)}  \log N \le N^{s(i_0)+(i_0+1)(d-\sigma-i_0)}  \log N.
\end{equation}

It remains to observe that the definition of the function $\cl{x}$ implies 
$\cl{x-1/2}\le x$, thus we obtain $i_0\le d-\sigma$ is within the range under 
consideration.  

{\bf Case 2.} Suppose that $i>d-\sigma$. Then from~\eqref{eq:H} and~\eqref{eq:here} we obtain 
$$
\varXi_N\ll N^{s(i)+i(d-\sigma-i)}.
$$
Let $g(t)=t(t+1)/2+t(d-\sigma-t)$ then $g$ attains the maximal value at 
$$
u_0=d-\sigma+1/2.
$$
As before, it now follows that for each $i=0, 1, \ldots, d$ we have 
\begin{equation}
\label{eq:case2}
\varXi_N\ll N^{s(i)+i(d-\sigma-i)}\le N^{s(j_0)+j_0(d-\sigma-j_0)},
\end{equation}
where 
\begin{equation}
\label{eq:j0}
j_0= \cl{u_0}= \cl{d-\sigma+1/2}.
\end{equation}
Again the definition of $\cl{x}$ gives $\cl{x+1/2}>x$ ,
thus we  see that $j_0>d-\sigma$ is an admissible value.

Combining~\eqref{eq:case1} and~\eqref{eq:case2}, we derive that  
\begin{equation}
\label{eq:case-case}
\varXi_N \ll N^{s(i_0)+(i_0+1)(d-\sigma-i_0)}+N^{s(j_0)+j_0(d-\sigma-j_0)},
\end{equation}
where $i_0, j_0$ are given at~\eqref{eq:i0} and~\eqref{eq:j0}, respectively.  

Observe that the definition of $\cl{x}$ implies that 
$$
\cl{x+1/2}=\cl{x-1/2}+1,
$$
and thus $j_0=i_0+1$ and one now verifies that 
$$
s(i_0)+(i_0+1)(d-\sigma-i_0) = s(j_0)+j_0(d-\sigma-j_0), 
$$  
hence the two terms in~\eqref{eq:case-case} are equal. Therefore,  we have
$$
\varXi_N\ll N^{s(j_0)+j_0(d-\sigma-j_0)},
$$
where $j_0$ is given by~\eqref{eq:j0}. Thus by~\eqref{eq:A} we obtain the desired bound.

\subsection{Proof of Theorem~\ref{thm:MVT-higher}}
Applying a similar chain of arguments as in the proof of Theorem~\ref{thm:MVT-higher-box}, substituting   the bound~\eqref{eq:zero-max} 
in the inequality 
$$
\int_{\Tor}   \left|\sum_{n=1}^{N} a_n\e(x_1 n+\ldots +x_d n^d)\right|^{2 s(d)} d \mu(\vx)
\le \sum_{\bxi\in \Z^{d}} J_{\va, N}(\bxi)\left |\widehat{\mu}(\bxi)\right|,
$$
see~\eqref{eq:A} and recalling   the condition~\eqref{eq:L1} we obtain the desired bound.  

\subsection{Proof of Theorem~\ref{thm:almostall}}

Let  $\alpha > \vartheta$  and set
$$
N_i =    2^i, \qquad i =1, 2, \ldots.
$$

Recall that the set $W_{\va, d}(\vx; N)$ is given by Lemma~\eqref{lem:completion}.  We now consider the set 
$$
\cB_{i} = \left\{ \vx\in \spt \mu:~ \left|W_{\va, d}( \vx; N_{i})\right|\ge N_i^{\alpha}  \right\}. 
$$

By the {\it  Markov inequality}, the definition of $W_{\va, d}(\vx; N_i)$ and  the mean value bound of Theorem~\ref{thm:almostall},  we derive 
$$
\mu(\cB_i)\le N_i^{-\rho \alpha}\int_{\Tor}\left |W_{\va, d}(\vx; N_i)\right|^\rho d\mu(\vx)\le N_i^{\rho \vartheta-\rho \alpha+o(1)}.
$$
Combining with  $\alpha>\vartheta$ and $N_i=2^{i}$, we have 
$$
\sum_{i=1}^{\infty}\mu(\cB_i)<\infty.
$$
Thus the  {\it Borel--Cantelli lemma\/} implies  
$$
\mu \left(\bigcap_{q=1}^{\infty}\bigcup_{i=q}^{\infty} \cB_i \right)=0.
$$
It follows that  for $\mu$-almost all $\vx\in \spt \mu$ there exists $i_{\vx}$ such that for any $i\ge i_{\vx}$  one has  
\begin{equation}
\label{eq:y}
|W_{\va, d}(\vx; N_i)|\le N_i^{\alpha}.
\end{equation}
We now fix this $\vx$ in the following argument.   

For any $N\ge N_{i_{\vx}}$ there exists $i$ such that 
$$
N_{i-1}\le N< N_{i}.
$$
By Lemma~\ref{lem:completion} and~\eqref{eq:y} we have 
$$
S_{\va, d}(\vx; N) \ll W_{\va, d}(\vx; N_i) \ll N^{\alpha}.
$$
Since $\alpha>\vartheta$ is arbitrary, this  gives the desired result. 

\subsection{Proof of Theorem~\ref{thm:HD}}

We start from some auxiliary results.  
We adapt the definition of the \emph{singular value  function} from~\cite[Chapter~9]{Falconer} to the following. 

\begin{definition} 
Let $\cR\subseteq \R^d$ be a rectangle with side lengths 
$$
r_1\ge \ldots \ge r_d.
$$ 
For $0< t \le d$  we set  $$
\varphi_{0, t}(\cR)=r_1^{t},
$$
and for $k=1, \ldots, d-1$ we define 
$$
\varphi_{k, t}(\cR)=r_1\ldots r_{k}r_{k+1}^{t-k}.
$$
\end{definition}

Note that  for a rectangle $\cR\subseteq \R^2$ with the side length $r_1\ge r_2$ we have 
$$
\varphi_{k, t} (\cR)=
  \begin{cases}
   r_1^t  & \text{ for } k=0, \\
   r_1 r_2^{t-1} &  \text{ for }  k=1.
  \end{cases}
$$

\begin{remark}
The notation  $\varphi_{k, t}(\cR)$ roughly means that we can cover the rectangle $\cR$ by 
about (up to a constant factor) 
$$
\frac{r_1}{r_{k+1}}\ldots \frac{r_k}{r_{k+1}}
$$
balls  of radius
$ r_{k+1}$, and hence this leads to the term 
$$
\varphi_{k, t}(\cR) = \frac{r_1}{r_{k+1}}\ldots \frac{r_k}{r_{k+1}} r_{k+1}^t
$$
in the expression for the Hausdorff measure with the parameter $t$ 
(again up to a constant factor which does not affect our results). 
\end{remark}

We now turn to the proof of Theorem~\ref{thm:HD}.  From the definition of the Hausdorff dimension, using the above notation and applying  Corollary~\ref{cor:cover} we obtain
\begin{equation}\label{eq:computing-HD}
\begin{split}
\dim (\cE_{\va, d, \alpha+\eta} \cap \spt \mu) \le \inf \Bigl \{ t>0:&~\sum_{i=1}^{\infty} \sum_{\cR\in \fR(N_i)}  \varphi_{k, t}(\cR)<\infty, \\
&  \text{for some }  k=0, \ldots, d-1 \Bigr \}.
\end{split}
\end{equation}

Furthermore, for  $k=1, \ldots, d-1$ and $0<t\le d$ we have 
\begin{equation} 
\begin{split}
\label{eq:i-level}
\sum_{\cR\in \fR(i)} \varphi_{k, t}(\cR)&=\# \fR(N_i)\zeta_{i, k+1}(\varepsilon)^{t-k} \prod_{j=1}^{k} \zeta_{i, j}(\varepsilon)   \\
&\le N_i^{\rho \vartheta-\rho \alpha} f\(\bzeta_i(\varepsilon)\)^{-1}\zeta_{i, k+1}(\varepsilon)^{t-k} \prod_{j=1}^{k} \zeta_{i, j}(\varepsilon)\\
&\le N_i^{\rho \vartheta -\rho \alpha+s(k)+t(\alpha-k-2)+C(d)\varepsilon} f\(\bzeta_i(\varepsilon)\)^{-1}, 
\end{split}
\end{equation}
where $C(d)$  is a positive constant which depends only $d$. We remark that~\eqref{eq:i-level}  also holds for the case $k=0$, in which we have  $s(k)=0$.  
To be precise for $k=0$ we have 
$$
\sum_{\cR\in \fR(N_i)} \varphi_{0, t}(\cR)\le N_i^{\rho \vartheta -\rho \alpha+t(\alpha-2)+C(d)\varepsilon} f\(\bzeta_i(\varepsilon)\)^{-1}.
$$

Applying~\eqref{eq:computing-HD} we derive that  
$$
\dim (\cE_{\va, d, \alpha+\eta} \cap \spt \mu) \le t
$$
provided that the parameters  $\alpha, \rho,  k, t, \vartheta$ satisfy the following  condition
$$
\sum_{i=1}^{\infty}N_i^{\rho \vartheta -\rho \alpha+s(k)+t(\alpha-k-2)+C(d)\varepsilon} f\(\bzeta_i(\varepsilon)\)^{-1}<\infty,
$$
where $c(d)$ is a positive constant that depends only  on $d$. By the arbitrary choice of $\eta>0$ we finish the proof.

\section{Proofs of  results for special sets}

\subsection{Proof of  Theorem~\ref{thm:MVT-Mom}}

We start with recalling the following well-known estimate, see, for example,~\cite[Equation~(8.6)]{IwKow}.

\begin{lemma}
\label{lem:linear-bound}
For any $t\in [-1/2, 1/2]$ we have (for convenience we set $1/0=\infty$) 
$$
\sum_{n=1}^{N} \e(nt)\ll \min \left \{N, \frac{1}{t}\right \}.
$$
\end{lemma}

For interval $[a, b]\subseteq [0,1]$ let $\mu_{[a, b]}$ be the natural measure  on the moment curve over the interval $[a, b]$, see~\eqref{eq:moment-measure}.

We also need the following $L^2$-type estimate which could be of independent interest.

\begin{lemma}
\label{lem:L2}
For any  $0<\delta<1/2$ and $N\in \N$  such that $N\delta>1$ we have 
$$
\sum_{\bxi\in \cD_N} |\widehat{\mu_{[\delta, 1/2]}}(\bxi)|^{2}\ll \delta^{1-d} N^{s(d)-d},
$$
where the notation $\cD_N$ is given by~\eqref{eq:DN}.
\end{lemma}
\begin{proof}
Observe that 
\begin{align*}
\sum_{\bxi\in \cD_N} & |\widehat{\mu_{[\delta, 1/2]}}(\bxi)|^{2}=   \sum_{\bxi\in \cD_N} \left |\int_{\delta}^{1/2} \e(\xi_1 t+\ldots+\xi_d t^{d})dt \right |^{2}\\
&\qquad = \sum_{\bxi\in \cD_N} \int_{\delta}^{1/2}\int_{\delta}^{1/2} \e(\xi_1(u-v)+\ldots+\xi_d (u^d-v^{d}))dudv\\
&\qquad = \int_{\delta}^{1/2}\int_{\delta}^{1/2} \prod_{i=1}^{d} \sum_{|\xi_i|\le 2s(d)N^{i}} \e(\xi_i(u^{i}-v^{i})) du dv.
\end{align*}  

For each $i=1, \ldots, d$ the Lagrange mean value theorem  implies that 
$$
i\delta^{i-1}|u-v|\le |u^{i}-v^{i}|\le 1/2.
$$
Combining with Lemma~\ref{lem:linear-bound} we arrive at
\begin{equation} \label{eq:hat-mu-2}
\begin{split}
\sum_{\bxi\in \cD_N}& |\widehat{\mu_{[\delta, 1/2]}}(\bxi)|^{2}\\
&\quad \ll \int_{\delta}^{1/2}\int_{\delta}^{1/2} \prod_{i=1}^{d} 
\min\left\{N^{i}, \frac{1}{|u^{i}-v^{i}|}\right\} du dv\\
&\quad \ll \delta^{d-s(d)} \int_{\delta}^{1/2}\int_{\delta}^{1/2} \prod_{i=1}^{d} \min\left\{N^{i}\delta^{i-1}, \frac{1}{|u-v|}\right\}  du dv.
\end{split} 
\end{equation}

Now we decompose the set $\cS=[\delta,1/2]\times [\delta, 1/2]$ into finite ``strips" and estimate the above integral over each strip. Precisely, for each $i=1, \ldots, d-1$, denote 
$$
\cS_i=\left\{(u, v)\in \cS:~N^{i}\delta^{i-1}\le \frac{1}{|u-v|}<N^{i+1}\delta^{i}\right\}.
$$
Furthermore, for $i=0$ and $i=d$  denote 
\begin{align*}
\cS_0&=\left\{(u, v)\in \cS: ~\frac{1}{|u-v|}<N\right\} ,\\
\cS_d&=\left\{(u, v)\in \cS:~\frac{1}{|u-v|}\ge N^{d}\delta^{d-1}\right\} .
\end{align*}

{\bf Integration over $\cS_0$.} Since $N\delta>1$,  any $(u, v)\in \cS_0$ implies 
$$
\frac{1}{|u-v|}<N^{i}\delta^{i-1}, \quad i=1, \ldots, d.
$$
Thus we obtain 
\begin{equation}\label{eq:S-0}
\begin{split}
 \int_{\cS_0}  \prod_{i=1}^{d} \min&\left\{N^{i}\delta^{i-1}, \frac{1}{|u-v|}\right\}  du dv\\
 &  \le \int_{\cS_0} \frac{1}{|u-v|^{d}} du dv\\
 & \ll \sum_{k=1}^{\log N} 2^{kd} \lambda(\{(u, v)\in \cS_0:~2^{-k-1}<|u-v|\le 2^{-k}\})\\
 &  \ll N^{d-1}, 
\end{split} 
\end{equation}  
where $\lambda$ is the  $2$-dimensional Lebesgue measure.
 
{\bf Integration over $\cS_d$.} Again since $N\delta>1$,  any $(u, v)\in \cS_d$ implies 
$$
\frac{1}{|u-v|}\ge N^{i}\delta^{i-1}, \qquad i=1, \ldots, d.
$$
Thus we have 
\begin{equation}\label{eq:S-d}
\begin{split}
 \int_{\cS_d} \prod_{i=1}^{d} \min\left\{N^{i}\delta^{i-1}, \frac{1}{|u-v|}\right\}  du dv & \le \int_{\cS_d} N^{s(d)} \delta^{s(d)-d} du dv\\
 & \ll N^{s(d)-d}\delta^{s(d)-2d+1}.
\end{split} 
\end{equation}

{\bf Integration over  $\cS_i$.}  For $i=1, \ldots, d-1$ the definition of $\cS_i$ implies 
$$
\prod_{i=1}^{d} \min\left\{N^{i}\delta^{i-1}, \frac{1}{|u-v|}\right\}\le  N^{s(i)}\delta^{s(i)-i} \frac{1}{|u-v|^{d-i}}.
$$
It follows that 
\begin{equation}\label{eq:Si}
\begin{split}
 \int_{\cS_i} \prod_{i=1}^{d}  \min & \left\{N^{i}\delta^{i-1}, \frac{1}{|u-v|}\right\}  du dv\\
 & \qquad  \le N^{s(i)}\delta^{s(i)-i}  \int_{\cS_i}  \frac{1}{|u-v|^{d-i}}dudv. 
\end{split} 
\end{equation}
Taking dyadic decomposition over the range 
$$
N^{i}\delta^{i-1}\le 1/|u-v|< N^{i+1}\delta^{i},
$$
that is for each $k\in \N$ let 
$$
\cS_{i, k}=\{(u, v)\in \cS_i: 2^{k-1} N^{i}\delta^{i-1}\le 1/|u-v|<2^{k} N^{i}\delta^{i-1}\}.
$$
Then  for the  Lebesgue measure of $\cS_{i, k}$ we have    
$$
\lambda\(\cS_{i, k}\) \ll 2^{-k}N^{-i}\delta^{1-i}.
$$ 
Thus we derive 
\begin{align*}
\int_{\cS_i}  \frac{1}{|u-v|^{d-i}}dudv &\ll \sum_{k=1}^{\log N\delta}(2^{k}N^{i}\delta^{i-1})^{d-i} \lambda\(\cS_{i, k}\)\\
&\ll \sum_{k=1}^{\log N\delta}(2^{k}N^{i}\delta^{i-1})^{d-i} 2^{-k}N^{-i}\delta^{1-i} \\
&= N^{(d-i-1)(i+1)}\delta^{i(d-i-1)}.
\end{align*}  
Combining with~\eqref{eq:Si} we obtain 
\begin{equation}\label{eq:SSi}
\begin{split}
 \int_{\cS_i} \prod_{i=1}^{d} \min&\left\{N^{i}\delta^{i-1}, \frac{1}{|u-v|}\right\}  du dv\\
 &\qquad \le N^{(i+1)(d-1-i/2)}\delta^{(d-1)i-s(i)}. 
\end{split}
\end{equation} 

The function $f(t)=(t+1)(d-1-t/2)$ attains its maximal value at 
$$
t_0=d-3/2.
$$ 
Note that the function $g(t)=(d-1)t-s(t)$  attains its minimal  value at $t_0=d-3/2$.

Let  $i_0=d-1$ (or $i_0=d-2$ by symmetry). Substituting in~\eqref{eq:SSi} we obtain 
\begin{equation}\label{eq:S-i}
\begin{split}
 &\int_{\cS_i} \prod_{i=1}^{d} \min\left\{N^{i}\delta^{i-1}, \frac{1}{|u-v|}\right\}  du dv\ll N^{s(d)-d} \delta^{s(d)-2d+1}. 
\end{split}
\end{equation}

Combining~\eqref{eq:hat-mu-2} with the estimates~\eqref{eq:S-0},~\eqref{eq:S-d} and~\eqref{eq:S-i} we arrive at
$$
\sum_{\bxi\in \cD_N} |\widehat{\mu_{[\delta, 1/2]}}(\bxi)|^{2}\ll \delta^{d-s(d)}N^{d-1}+\delta^{1-d} N^{s(d)-d}\ll \delta^{1-d}N^{s(d)-d},
$$
 since $N\delta>1$ and $d\ge 2$. 
\end{proof}

The proof of Lemma~\ref{lem:L2} implies the following result.  Recall that for an
interval $[a, b]\subseteq [0,1]$ the measure  $\mu_{[a, b]}$ is  the natural measure  on the moment curve over the interval $[a, b]$, see~\eqref{eq:moment-measure}. 

\begin{lemma}
\label{lem:short-interval}
For any interval $[a, b]\subseteq [1/2, 1]$ with $b-a\le 1/2d$ we have 
$$
\sum_{\bxi\in \cD_N} |\widehat{\mu_{[a, b]}}(\bxi)|^{2}\ll N^{s(d)-d}.
$$
\end{lemma}
\begin{proof}
For each $i=1, \ldots, d$ and any $u, v\in [a, b]$ the Lagrange mean value theorem  implies that 
$$
|u-v|\ll |u^{i}-v^{i}|\le 1/2.
$$
Thus applying the argument as in the proof of Lemma~\ref{lem:L2}, taking $\delta$ to be some positive constant,  we obtain the desired bound.
\end{proof}

We now turn to the proof of Theorem~\ref{thm:MVT-Mom}. Without loss of generality we assume that $0<\delta<1/2$.   
First of all let 
\begin{equation}
\label{eq:decomposition}
[\delta, 1]\subseteq \bigcup_{j=0}^{2d} \cI_j, 
\end{equation}
where 
$\cI_0=[\delta, 1/2]$ and  for $j=1, \ldots, 2d$
$$
\cI_j= [1/2+(j-1)/2d, 1/2+j/2d].
 $$
 Note that for each interval $\cI\subseteq [0, 1]$, applying the argument as in the proof of Theorem~\ref{thm:MVT-higher-box},  
 similarly to~\eqref{eq:A}  we obtain  
\begin{equation}
\label{eq:2s2s}
\int_{\cI} \left |  \sum_{n=1}^{N}\e(tn+\ldots+t^{d}n^{d})\right |^{2s} dt\le  
\sum_{\bxi\in \cD_N}I_s(N; \bxi) \left |  \widehat{\mu_{\cI}}(\bxi) \right |, 
\end{equation}
where $I_s(N;\bxi)$ is defined similarly to $J_{\va, N}(\bxi)$ as 
the  number of solutions to the system of equations
\begin{align*}
\sum_{j=1}^{s}  n_j^i & - \sum_{j=s +1}^{2s }  n_j^i =\xi_i,   \qquad  i=1, \ldots, d, \\
 n_j  &= 1, \ldots, N,  \qquad   j =1, \ldots, 2s. 
\end{align*}

Combining with~\eqref{eq:decomposition} and~\eqref{eq:2s2s} we obtain 
\begin{equation}\label{eq:delta-1}
\begin{split}
&\int_{\delta}^{1} \left |  \sum_{n=1}^{N}\e(tn+\ldots+t^{d}n^{d})\right |^{2s } dt\\
&\qquad\qquad\le \sum_{j=0}^{2d} \int_{\cI_j} \left |  \sum_{n=1}^{N}\e(tn+\ldots+t^{d}n^{d})\right |^{2s } dt\\
&\qquad \qquad \le \sum_{j=0}^{2d} \sum_{\bxi\in \cD_N}I_s(N; \bxi) \left |  \widehat{\mu_{\cI_j}}(\bxi) \right |.
\end{split}
\end{equation}

By the Cauchy-Schwarz inequality we have
\begin{equation}
\label{eq:C-S}
 \sum_{\bxi\in \cD_N} I_s(N; \bxi)  |\widehat{\mu_{\cI_j}}(\bxi)|\ll \left ( \sum_{\bxi\in \cD_N}I_s(N; \bxi)^ 2\sum_{\bxi\in \cD_N} |\widehat{\mu_{\cI_j}}(\bxi)|^{2} \right )^{1/2}. 
\end{equation}
We now note that by~\eqref{eq:MVT}  we obtain
\begin{equation}\label{eq:ISN}
 \begin{split}
 \sum_{\bxi\in \cD_N}I_s(N; \bxi)^2 & \le  I_{2s}(N; \mathbf{0})  = \int_{\Tor} |S_d(\vx; N)|^{4s}d\vx\\
 &  \le N^{2s +o(1)} +  N^{4s - s(d)+o(1)}.
\end{split}
\end{equation}

For $j=0$, applying~\eqref{eq:C-S},~\eqref{eq:ISN} and  Lemma~\ref{lem:L2},  we arrive at 
\begin{equation}\label{eq:delta-1/2}
\begin{split}
 \sum_{\bxi\in \cD_N} I_s(N; \bxi) & |\widehat{\mu_{[\delta, 1/2]}}(\bxi)| \\
 &\ll    \delta^{(1-d)/2}  \(N^{s} +  N^{2s - s(d)/2}\)N^{s(d)/2-d/2+o(1)}.
\end{split}
\end{equation}

Similarly,  for each $j=1, \ldots, 2d$, we use Lemma~\ref{lem:short-interval} instead of Lemma~\ref{lem:L2} and   derive 
$$
 \sum_{\bxi\in \cD_N} I_s(N; \bxi) |\widehat{\mu_{\cI_j}}(\bxi)|\ll   \(N^{s} +  N^{2s - s(d)/2}\)N^{s(d)/2-d/2+o(1)}.
$$
Combining this with~\eqref{eq:delta-1} and~\eqref{eq:delta-1/2} we obtain the desired bound.

\subsection{Proof of  Theorem~\ref{thm:moment}}
Let $\delta>0$ and 
$
\Gamma_\delta=\Gamma\setminus \cB(\v0, \delta).
$
Moreover let $\muMd$ be the natural measure  on $\Gamma_{\delta}$, see~\eqref{eq:moment-measure}. 
 For any  $\bzeta=(\zeta_1, \ldots, \zeta_d)$ with 
$0<\zeta_j<1$, $j=1, \ldots, d$, we have 
\begin{equation}
\label{eq:intersection-Gamma-delta}
\zeta_d\ll \diam(\cR(\vx, \bzeta)\cap \Gamma_\delta)\ll\zeta_d,
 \quad \vx\in \Gamma_\delta, 
\end{equation} 
where the implied constant may 
depend on $\delta$.  
It follows that the measure $\muMd$ is $f$-regular if we take 
\begin{equation}
\label{eq:f-m}
f(\bzeta)=c_d\, \zeta_d
\end{equation}
for some positive constant $c_d$.

Note that Lemma~\ref{lem:moment-decay} implies that 
$$
\widehat{\muMd}(\bxi)\ll  \(1+\|\bxi\|\)^{-1/d}, \qquad \bxi\in \R^d.
$$
Thus by Remark~\ref{rem:almostall} we derive that the measure $\muMd$ satisfies the condition of Theorem~\ref{thm:HD} with $\rho=2$ and $\vartheta=1/2$.

Let $1/2<\alpha<1$ and let $\eps>0$ be a small parameter. In the following we use the notation from Corollary~\ref{cor:cover}.
Let $N_i=2^{i}$, $i\in \N$. Then for any $\eta>0$   we have   
$$
\cE_{\va, d, \alpha+\eta} \cap \Gamma_{\delta} \subseteq \bigcap_{q=1}^{\infty}\bigcup_{i=q}^{\infty} \bigcup_{\cR\in \fR(N_i)} (\cR\cap \Gamma_\delta).
$$
Moreover, by Corollary~\ref{cor:cover} the centre  of each $\cR\in \fR(N_i)$ is in $\Gamma_\delta$.  For $t>0$,  applying Corollary~\ref{cor:cover}, the upper bound in~\eqref{eq:intersection-Gamma-delta}, and~\eqref{eq:f-m}, we obtain 
\begin{align*}
\sum_{i=1}^{\infty} \# \fR(N_i) \diam(\cR\cap \Gamma_\delta)^t
&\le  N_{i}^{1-2 \alpha+o(1)} \zeta_{i, d}^{-1+t}\\
&\le N_i^{d+2-3\alpha+t(\alpha-d-1)+C(d)\eps+ o(1)}.
\end{align*}
where $C(d)>0$ is a  constant that depends only on $d$.  Applying~\eqref{eq:computing-HD}, we conclude that  
$$
\dim (\cE_{\va, d, \alpha+\eta} \cap \Gamma_\delta)\le t
$$
provided that the parameters   satisfy the following  further condition
$$
d+2-3\alpha+t(\alpha-d-1)+C(d)\eps<0.
$$
By the arbitrary choice of $\eps>0$ it is sufficient  to have 
$$
d+2-3\alpha+t(\alpha-d-1)<0.
$$
Thus we conclude that 
\begin{equation}
\label{eq:Gamma_delta}
\dim (\cE_{\va, d, \alpha+\eta}\cap \Gamma_\delta)\le \frac{d+2-3\alpha}{d+1-\alpha}.
\end{equation}

Note that Hausdorff dimension has the following {\it countable stability\/}  (see~\cite[Section~2.2]{Falconer}): 
for $\cA_i \subseteq \R^d$, $i\in \N$ we have 
\begin{equation}
\label{eq:stability}
\dim \bigcup_{i=1}^{\infty} \cA_i=\sup_{i\in \N} \dim \cA_i.
\end{equation}

Clearly we have  
$$
\Gamma=\bigcup_{i=1}^{\infty} \Gamma_{1/i} \cup \left\{0\right\}.
$$ 
Therefore, combining~\eqref{eq:Gamma_delta} and~\eqref{eq:stability} we derive that 
$$
\dim (\cE_{\va, d, \alpha+\eta}\cap \Gamma)\le \frac{d+2-3\alpha}{d+1-\alpha}.
$$
By the arbitrary choice of $\eta$ we obtain the desired bound.

\subsection{Proof of Theorem~\ref{thm:MVT-L}}

Let $\bomega=(\omega_1, \ldots, \omega_d)$ and 
$$
\bxi_{n, m}=(n-m, \ldots, n^d-m^d).
$$
Let 
$$
\fI=\int_{L_{\bomega}} |S_{\va, d}(\vx)|^{2}d\muo(\vx) 
=\int_{0}^{1} \left|\sum_{n=1}^{N} a_n \e(t\omega_1 n+\ldots t\omega_d n^d)\right|^2 dt. 
$$  
Expanding the square, we obtain 
$$
\fI =\int_{0}^{1}\sum_{1\le n, m\le N}a_n\overline{a_m}\e(t \bomega \cdot \bxi_{n-m}) dt.
$$

Observe that there exists a positive constant $C_{\bomega}$ such that if 
$$ 
\max\{n, m\}\ge C_{\bomega},
$$
and $n\neq m$, then we have 
\begin{equation}
\label{eq:nm}
|\bomega \cdot \bxi_{n-m}|\gg |n-m|.
\end{equation}  
Indeed, for $\bomega$ let $1\le k\le d$ be the maximal number such that $\omega_k\neq 0$. If $k=1$ then clearly we have~\eqref{eq:nm}. For the case $k>1$, for each $1\le i<k$ and $n\neq m$,  we have 
$$
\frac{|w_i(n^i-m^{i})|}{|w_k(n^{k}-m^{k})|}\ll \min\left \{\frac{1}{n}, \frac{1}{m}\right \}, 
$$
where the implied constant may 
depend on $\bomega$.  Thus by choosing $C_{\bomega}$ large enough, and $n\neq m$ we obtain 
$$
|\bomega \cdot \bxi_{n-m}| \gg  |w_k(n^{k}-m^{k})|\gg |n-m|,
$$
which gives~\eqref{eq:nm}.

Applying  Lemma~\ref{lem:var} with   
$k=1$ and the estimate~\eqref{eq:nm}, we deduce 
\begin{align*}
\fI
&\le \int_{0}^{1}\sum_{1\le n, m\le N}a_n\overline{a_m}\e(t \bomega \cdot \bxi_{n-m}) dt\\
&\ll \sum_{\substack{1\le n=m\le N \\ \text{ or }1\le n, m\le C_{\bomega}}} a_n\overline{a_m}+N^{o(1)}\sum_{1\le m \neq n\le N}  1/|n-m|\\
& \le N^{1+o(1)},
\end{align*}
which gives the desired result.

\subsection{Proof of Corollary~\ref{cor:almostall-segment}}
Let $a(N, n)$ be a double sequence with the condition~\eqref{eq:double a}. Taking $a(N, n)$ instead of $a_n$ in the proof of Theorem~\ref{thm:MVT-L}, we obtain 
$$
\int_{0}^{1} \left|\sum_{n=1}^{N} a(n, N) \e(t\omega_1 n+\ldots t\omega_d n^d)\right|^2 dt\le N^{1+o(1)}.
$$
Using Theorem~\ref{thm:almostall} we obtain the desired result.

\subsection{Proof of Corollary~\ref{cor:HD-segment}}

Let $1/2<\alpha<1$ and let $\eps>0$ be a small parameter. In the following we use the notation from Corollary~\ref{cor:cover}.
Let $N_i=2^{i}$, $i\in \N$. Then for any $\eta>0$   we have 
$$
\cE_{\va, d, \alpha+\eta} \cap L_{\bomega} \subseteq \bigcap_{q=1}^{\infty}\bigcup_{i=q}^{\infty} \bigcup_{\cR\in \fR(N_i)} (\cR\cap L_{\bomega}).
$$

Note that the centre  of each $\cR\in \fR(N_i)$ is in $L_{\bomega}$. Thus the assumption~\eqref{eq:omega-k} implies that 
\begin{equation}
\label{eq:diam-L}
\zeta_{i, k}\ll\diam(\cR\cap L_{\bomega})\ll  \zeta_{i, k}.
\end{equation} 

By applying arguments   similar to that in the proof of Theorem~\ref{thm:moment}, taking $\zeta_{i, k}$ 
instead of $\zeta_{i, d}$,  we obtain the desired bound.   

\begin{remark}
\label{rem:unif segm} 
Note that for  segments we have the uniform bound~\eqref{eq:diam-L}, thus there is not need  to use the decomposition argument as in the proof of Theorem~\ref{thm:moment}.
\end{remark}

 \section{Comments}
 Certainly the method of the proof of  Theorem~\ref{thm:MVT-Mom} works for many other 
 polynomial curves and rationally parametrised varieties, that is,  for exponential sums 
 with polynomials 
$$
  f_\vt(X) = g_1(\vt) X+ \ldots +  g_d(\vt) X^d, \quad \vt = (t_1, \ldots, t_m) \in  \R^m
$$
where $g_i({\mathbf T}) \in \R[{\mathbf T}]$, $i=1, \ldots, d$,  are polynomials in $m$ variables, 
although the specific estimate in  Lemma~\ref{lem:L2} depends on the specific 
form of the moment curve~\eqref{eq:moment-curve}.

However we do not see any approach to improving the general bound of 
Theorem~\ref{thm:MVT}   and  Theorem~\ref{thm:MVT-higher-box} for 
the parameter $\vx$ which runs through some general algebraic variety $\cV$,  
that is, for the integrals  
$$
I_s(\cV) =  \int_{\cV}   \left|\sum_{n=1}^{N} a_n\e\(x_1n+\ldots +x_d n^{d}\) \right|^s d \mu_\mathsf{V}(\vx), 
$$
where $\mu_\mathsf{V}$ is some natural  measure on $\cV$.

Note that Example~\ref{exam:MVT S M} gives upper bounds on $I_{2s(d)}\(\Sp^{d-1}\)$, 
which follow directly from  Theorem~\ref{thm:MVT-higher-box}. We are however interested 
in stronger results utilising  some specific properties of $\cV$, which we pose as an 
open question. 

 \section*{Acknowledgement}

The authors are grateful to Bryce Kerr for his encouragement and many helpful discussions. 
 
This work was  supported   by ARC Grant~DP170100786.

\end{document}